\documentclass[11pt]{amsart}

\usepackage[dvipsnames]{xcolor}
\usepackage{amssymb,verbatim}
\usepackage{amsmath,amsfonts,enumitem,bm}
\usepackage[mathscr]{euscript} 
\usepackage{amsthm}
\usepackage{url}
\usepackage{graphicx} 
\usepackage{enumitem} 
\usepackage{hyperref} 
\hypersetup{colorlinks} 
\usepackage{bbm} 
\usepackage{bm} 
\usepackage{mathrsfs} 
\usepackage{cancel} 
\usepackage{ytableau} 
\usepackage{mathtools}

\usepackage[colorinlistoftodos]{todonotes}

\RequirePackage{cleveref}
\usepackage{hypcap}
\hypersetup{colorlinks=true, citecolor=darkblue, linkcolor=darkblue}
\definecolor{darkblue}{rgb}{0.0,0,0.7}
\newcommand{\darkblue}{\color{darkblue}}

\definecolor{darkred}{rgb}{0.68,0,0}
\newcommand{\darkred}{\color{darkred}}

\definecolor{darkgreen}{rgb}{0,.38,0}
\newcommand{\darkgreen}{\color{darkgreen}}

\newcommand{\defn}[1]{\emph{\darkblue #1}}
\newcommand{\defna}[1]{\emph{\darkred #1}}

\newcommand{\defng}[1]{\emph{\darkgreen #1}}

\setlist[enumerate]{
	label=\textnormal{({\roman*})},
	ref={\roman*}}

\makeatletter
\def\th@plain{%
	\thm@notefont{}
	\itshape 
}
\def\th@definition{%
	\thm@notefont{}
	\normalfont 
}
\makeatother

\newtheorem{thm}{Theorem}[section]
\newtheorem{lemma}[thm]{Lemma}

\newtheorem*{claim*}{Claim}

\newtheorem{prop}[thm]{Proposition}
\newtheorem{conj}[thm]{Conjecture}
\newtheorem{conv}[thm]{Convention}

\theoremstyle{definition}

\newtheorem{rem}[thm]{Remark}

\numberwithin{figure}{section}
\numberwithin{equation}{section}

\def\bu{\bullet}
\def\emp{\nothing}

\def\zz{\mathbb Z}
\def\nn{\mathbb N}
\def\cc{\mathbb C}

\def\rr{\mathbb R}
\def\qqq{\mathbb Q}

\def\kk{\mathbb K}

\def\ov{\oa}

\def\la{\lambda}
\def\ga{\gamma}
\def\si{\sigma}
\def\de{\delta}
\def\ep{\ve}
\def\al{\alpha}
\def\be{\beta}
\def\om{\omega}

\def\ve{\varepsilon}

\def\ssu{\subset}

\def\<{\langle}
\def\>{\rangle}

\def\SO{ {\text {\rm SO} } }

\def\GL{ {\text {\rm GL} } }

\def\rR{ {\textsc {\rm R} } }

\def\oa{\overrightarrow}

\def\0{{\mathbf 0}}

\def\ol#1{{\overline {#1}}}

\def\nothing{\varnothing}

\def\.{\hskip.06cm}
\def\ts{\hskip.03cm}

\def\pt{\partial}

\def\bz{{\textbf{\textit{z}}}}

\def\bx{{\textbf{\textit{x}}}}

\def\by{{\textbf{\textit{y}}}}

\def\bbe{\textbf{\textit{e}}}

\def\bal{{\boldsymbol{\alpha}}}
\def\bbe{{\boldsymbol{\be}}}
\def\bpi{{\boldsymbol{\pi}}}
\def\brho{{\boldsymbol{\rho}}}

\def\La{\Lambda}

\def\ze{{\zeta}}

\newcommand{\sign}{\mathrm{sign}}

\newcommand{\SSYT}{\operatorname{SSYT}}

\def\.{\hskip.06cm}
\def\ts{\hskip.03cm}

\def\nin{\noindent}

\newcommand{\textsu}[1]{\textup{\textsf{#1}}}

\newcommand{\ComCla}[1]{\textup{\textsu{#1}}}

\newcommand{\sharpP}{\ComCla{\#P}}
\newcommand{\SP}{\ComCla{\#P}}
\newcommand{\SBQP}{\ComCla{\#BQP}}

\newcommand{\GapP}{\ComCla{GapP}}

\newcommand{\GapPP}{\ComCla{GapP}_{\ge 0}}

\newcommand{\Sigmap}{\ensuremath{\Sigma^{{\textup{p}}}}}
\newcommand{\Pip}{\ensuremath{\Pi^{{\textup{p}}}}}

\newcommand{\NP}{\ComCla{NP}}
\newcommand{\VP}{\ComCla{VP}}
\newcommand{\VNP}{\ComCla{VNP}}
\newcommand{\BPP}{\ComCla{BPP}}
\newcommand{\BQP}{\ComCla{BQP}}

\newcommand{\coNP}{\ComCla{coNP}}
\newcommand{\E}{\ComCla{E}}
\renewcommand{\P}{\ComCla{P}}
\newcommand{\CeqP}{\ComCla{C$_=$P}}

\newcommand{\PH}{\ComCla{PH}}
\newcommand{\CH}{\ComCla{CH}}
\newcommand{\PSPACE}{\ComCla{PSPACE}}
\newcommand{\FPSPACE}{\ComCla{FPSPACE}}
\newcommand{\EXPSPACE}{\ComCla{EXPSPACE}}
\newcommand{\FP}{\ComCla{FP}}
\newcommand{\PP}{\ComCla{PP}}
\newcommand{\NQP}{\ComCla{NQP}}
\newcommand{\QMA}{\ComCla{QMA}}
\newcommand{\coQMA}{\ComCla{coQMA}}

\newcommand{\AM}{\ComCla{AM}}
\newcommand{\coAM}{\ComCla{coAM}}
\newcommand{\MA}{\ComCla{MA}}

\newcommand{\EBPP}{\ComCla{$\exists\cdot$BPP}}

\def\SP{\sharpP}

\def\GRH{\textup{\sc GRH}}
\def\ERH{\textup{\sc ERH}}
\def\HN{\textup{\sc HN}}
\def\HNP{\textup{\sc HNP}}
\def\SHN{\textup{\sc \#HN}}

\def\SC{\textup{\sc Schubert}}
\def\SV{\textup{\sc SchubertVanishing}}
\def\LR{\textup{\sc LR}}

\def\poly{{\P}}
\def\CEP{{\CeqP}}
\def\coCEP{\ComCla{coC$_=$P}}

\newcommand{\inv}{\operatorname{{\rm inv}}}
\newcommand{\Des}{\operatorname{{\rm Des}}}
\newcommand{\des}{\operatorname{{\rm des}}}

\newcommand{\Sch}{\mathfrak{S}} 

\DeclareMathOperator{\zero}{\mathbf{0}} 

\newcommand{\RC}{{\text {\rm RC} } }

\newcommand{\fg}{\mathfrak{g}}

\newcommand{\Par}{\text{\bf Par}}
\newcommand{\VarB}{\text{\bf Var}}
\newcommand{\VarG}{\text{\bf Var}}
\newcommand{\EB}{\text{\bf Eq}}
\newcommand{\EG}{\text{\bf Eq}}

\begin{document}

\title[Vanishing of Schubert Coefficients]{Vanishing of Schubert Coefficients}

\author[Igor Pak \. \and \. Colleen Robichaux]{Igor Pak$^\star$  \. \and \.  Colleen Robichaux$^\star$
}

\makeatletter

\thanks{\thinspace ${\hspace{-.45ex}}^\star$Department of Mathematics,
UCLA, Los Angeles, CA 90095, USA. Email:  \texttt{\{pak,robichaux\}@math.ucla.edu}}


\thanks{\today}

\begin{abstract}
\emph{Schubert coefficients} are nonnegative integers \ts $c^w_{u,v}$ \ts
that arise in Algebraic Geometry and play a central role in Algebraic Combinatorics.
It is a major open problem whether they have a combinatorial interpretation,
i.e, whether \ts $c^w_{u,v} \in \SP$.

We study the closely related \emph{vanishing problem of Schubert coefficients}:
\ts $\{c^w_{u,v}=^?0\}$.  Until this work it was open whether
this problem is in the polynomial hierarchy~$\PH$.  We prove
that \ts $\{c^w_{u,v}=^?0\}$ \ts in $\coAM$ assuming the~$\GRH$.
In particular, the vanishing problem is in $\Sigmap_2\ts$.

Our approach is based on constructions \emph{lifted formulations},
which give polynomial systems of equations for the problem.
The result follows from a reduction to \emph{Parametric Hilbert's
Nullstellensatz}, recently studied in \cite{A+24}.
We extend our results to all classical types.
Type~$D$ is resolved in the appendix (joint with David Speyer).
%

\end{abstract}

\maketitle



\section{Introduction}\label{sec:intro}

\subsection{Foreword}\label{ss:intro-foreword}
This paper occupies an unusual ground,  bridging between
Schubert Calculus, a subarea of Algebraic Geometry and
Algebraic Combinatorics, on the one hand, and Computational
Complexity on the other.
Since these are far apart, we make an effort to give as much
background as we can to make the results comprehensible.

Below is a very lengthy introduction which includes a lot of
standard results and relevant background in Algebraic Combinatorics,
as well as their meaning in terms of computational complexity classes.
Eventually we get to the point when we can state our results and
explain their importance.  We beg the reader's forgiveness
for being introductory or even basic, at times.  The experts
can simply skip those parts.

Throughout the introduction we assume the reader is familiar
with standard notions and results in Computational Complexity.
The reader in need of a quick reminder is referred to
Section~\ref{s:CS}, where we give a quick overview of standard
complexity classes.

\smallskip

\subsection{Littlewood--Richardson coefficients}\label{ss:intro-back}
We begin with the Littlewood--Richardson (LR) coefficients
$\{c^\la_{\mu,\nu}\}$ that are the most important
special case of the Schubert coefficients.  The are also
a major source of inspiration behind many generalizations.

\defn{Schur polynomials}
\ts $s_\la=s_\la(x_1,\ldots,x_n) \in \nn[x_1,\ldots,x_n]$, are symmetric
polynomials corresponding to irreducible characters of $\GL_n(\cc)$ and are
given by an explicit determinant formula, see e.g.\
\cite{Mac95,Sta99}.
Here $\la$ is an \defn{integer partition} with at most $n$ parts:
$\la = (\la_1\ge \ldots \ge \la_n)$.   We use a shorthand \ts
$\bx := (x_1,\ldots,x_n)$.  Schur polynomials can also be defined as
$$
s_\la(\bx) \, := \, \sum_{A\in \SSYT(\la)} \. \bx^A \quad
\text{where} \quad \bx^A \, := \, \prod_{(i,j) \in \la} \. x_{A(i,j)}\..
$$
Here the summation is over \defn{semistandard Young tableaux} $A$ of shape $\la$,
i.e.\ integer functions $A: \la \to \nn$ on a Young diagram which weakly
increase in rows and strictly increase in columns, see Figure~\ref{f:SSYT}.

\begin{figure}[hbt]
\begin{center}
	\includegraphics[height=2.cm]{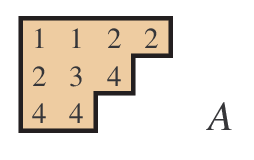}
\vskip-.3cm
\caption{Semistandard Young tableau \ts $A\in \SSYT(4,3,2)$ \ts and the
corresponding monomial \ts $\bx^A \ts = \ts x_1^2 \. x_2^3 \. x_3 \. x_4^3$\ts. }
\label{f:SSYT}
\end{center}
\end{figure}

From here it is easy to see that Schur functions form a linear basis
in the ring \ts $\La_n:=\cc[x_1,\ldots,x_n]^{S_n}$ \ts of symmetric polynomials
in~$n$ variables.  Therefore, their multiplication structure constants
in $\La_n$ are well defined.\footnote{Formally, one needs to take multiplication
in the inverse limit $\La$ of $\La_n$\ts,
so that these structure constants depend only
on partitions and not on~$n$.}
\defn{Littlewood--Richardson $($LR$)$ coefficients}
are now defined as:
$$s_\mu \cdot s_\nu \, = \, \sum_\la \, c^\la_{\mu,\nu} \, s_\la\ts.
$$

Schur polynomials encode fundamental symmetries in algebra and geometry.
They are ubiquitous in Enumerative and Algebraic Combinatorics
\cite{Ful97,Sta99}, Representation Theory \cite{Sag01},
Algebraic Geometry \cite{AF,Man01}, Quantum Information Theory \cite{OW21},
Algebraic Complexity \cite{CKLMSS}, and Geometric Complexity Theory (GCT)
\cite{BI18,IP17}.  Unsurprisingly, just about every aspect
of Schur polynomials and their generalizations has been studied
to great extent.

The LR-coefficients are especially well studied from both combinatorial and
computational points of view \cite{Pan23}.  There are at least~20 combinatorial
interpretations for $\{c^\la_{\mu,\nu}\}$,
some of which are transparently
$\SP$-functions for the unary input  and a few others for the binary input
\cite[$\S$11.4]{Pak-OPAC}.  Denote by \ts $\LR$ \ts the counting problem of
computing LR-coefficients.  It is known that computing LR-coefficients
is $\SP$-complete for the input in binary \cite{Nar06}.

\smallskip

\begin{conj}[{\rm \cite[Conj.~5.14]{Pan23}}{}]
\label{conj:LR-unary}
\ts $\LR$ \ts is \ts $\SP$-complete for input in unary.
\end{conj}

\smallskip

Most remarkably, the \emph{vanishing of LR-coefficients} $\{c^\la_{\mu,\nu}=^?0\}$
is in~$\poly$ even for the input in binary.  This was shown in \cite{DeLM06,MNS12}
as an easy consequence of the celebrated \defn{saturation theorem} by Knutson
and Tao \cite{KT99}.  See also \cite{BI13} for a fast algorithm based on
network max-flows.

This ``vanishing in $\poly$'' property inspired an
important part of Mulmuley's program GCT towards proving that $\VP\ne\VNP$ \cite{Mul09}.
Specifically, Mulmuley aimed to modify and extend this property to
\emph{Kronecker coefficients} which generalize the LR-coefficients,
see~$\S$\ref{ss:disc-comp}{\small $(iii)$}.
Some of Mulmuley's conjectures were refuted in \cite{BOR09},
while his general approach was undermined in \cite{BIP19,IP17}.  Most recently,
it was shown in \cite{PP20} that even \emph{reduced Kronecker coefficients}
do not have the saturation property, adding further doubts to the program.

\medskip

\subsection{Schubert polynomials}\label{ss:intro-Schub}
\defn{Schubert polynomials} \ts $\Sch_w\in \nn[x_1,\ldots,x_n]$ \ts indexed by
permutations,
are celebrated generalizations of Schur polynomials, and are not symmetric in general.
They were introduced by Lascoux and Sch\"{u}tzenberger~\cite{LS82,LS85}, to represent
cohomology classes of Schubert varieties in the complete flag variety  (see below).

Schubert polynomials have been intensely studied from algebraic, combinatorial,
and (more recently) complexity points of view.  See~\cite{Mac91,Man01} for
classic introductory surveys, \cite{Knu16,Knu22} for overviews of recent results,
\cite{AF,KM05} for geometric aspects, and \cite{Las95} for historical remarks.
We refer to \cite[$\S$10]{Pak-OPAC} for an overview of
computational complexity aspects.

We postpone the algebro-geometric background for Schubert polynomials
until~$\S$\ref{ss:setup-Schubert}.  We do not need it at this stage.
Instead, below we give two definitions.

\medskip

\nin
\underline{Algebraic Definition:} For a permutation \ts $w_\circ = (n,n-1,\ldots,2,1)$, let
$$
\Sch_{w_\circ} \. : = \. x_1^{n-1} x_2^{n-2} \. \cdots  \. x_{n-1}\..
$$
A permutation $w\in S_n$ is said to have a \defn{descent} at~$i$, if \ts $w(i)> w(i+1)$.
Denote by \ts $\Des(w)$ \ts the \defn{set of descents} \ts of~$w$, and by \ts $\des(\si):=|\Des(\si)|$
\ts the \defn{number of descents}. Define the \defn{divided difference operator}
$$
\pt_i F \, := \, \frac{F - s_i F }{x_i -x_{i+1}}\,,
$$
where the transposition $s_i:=(i,i+1)$ acts on $F\in \cc[x_1,\ldots,x_n]$ by transposing the variables.
For all $i \in \Des(w)$, let
$$
\Sch_{ws_i} \. := \. \pt_i \ts \Sch_w\.,
$$
and define all Schubert polynomials recursively.

This definition is due to Lascoux and Sch\"utzenberger \cite{LS82}.  It follows
that \ts $\Sch_w\in \zz[\bx]$ \ts are homogeneous polynomials of degree \ts $\inv(w)$.
Here \. $\inv(w):=\{(i,j) \, : \,i<j,\. w(i)>w(j)\}$ \. is the number of inversions in~$w$.

This definition can be generalized to other root systems
(also called \defn{types}). Indeed, in the corresponding \emph{Weyl group},
the transpositions $s_i$ are replaced by simple reflections, and $w_\circ$
is replaced by the longest element (\emph{long Weyl element}), see \cite{BH95}.
We omit the details which are both standard and well explained in \cite{AF}.

The disadvantage of this definition is a nonobvious combinatorial nature of
the coefficients.  It is known that  \ts $[\bx^\al] \Sch_w \in \nn$.
This was established in \cite{BJS93,FS94},
and follows from the following combinatorial definition due to Billey and
Bergeron \cite{BB93}.

\medskip

\nin
\underline{Combinatorial Definition:}
For a permutation \ts $w\in S_n$\ts, denote by $\RC(w)$  the set
of \ts \defn{RC-graphs} (also called \defn{pipe dreams}), defined as tilings
of a staircase shape with \defn{crosses} and \defn{elbows} as in the figure below,
such that:

{\small $(i)$} \, curves start in row $k$ on the left and end in column $w(k)$ on top, for all \ts $1\le k \le n$, and

{\small $(ii)$} \. no two curves intersect twice.

\nin
It follows from these conditions that every \ts $H \in \RC(w)$ \ts
has exactly \ts $\inv(w)$ \ts crosses.

\begin{figure}[hbt]
\begin{center}
	\includegraphics[height=2.5cm]{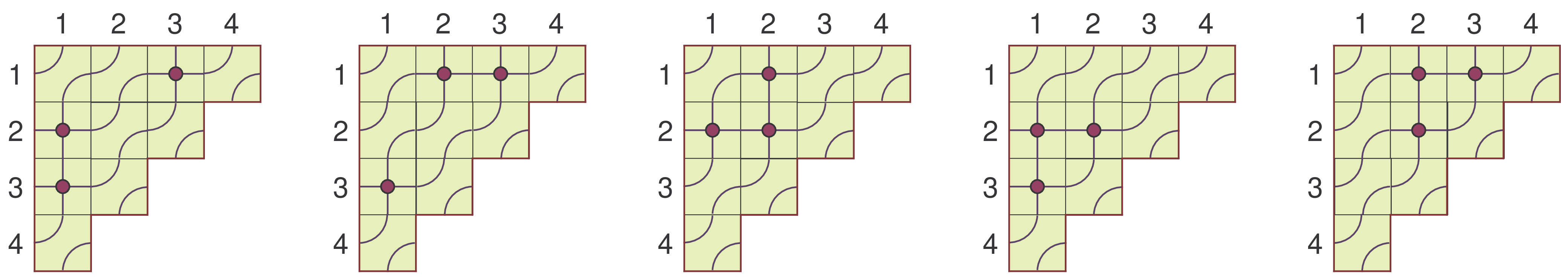}
\hskip-4.1cm
\caption{Graphs in \ts $\RC(1432)$ \ts and the corresponding Schubert
polynomial \.
 $\Sch_{1432} = x_1x_2x_3+x_1^2x_3+ x_1x_2^2+x_2^2x_3+x_1^2x_2$ \.
with monomials in this order.}
\label{f:RC}
\end{center}
\end{figure}


%
The \defn{Schubert polynomial} \ts $\Sch_w\in \nn[x_1,x_2,\ldots]$ \ts
is defined as
\begin{equation}\label{eq:Schubert-def}
\Sch_w(\bx) \, := \, \sum_{H\ts\in\ts \RC(w)} \. \bx^H \quad \text{where} \quad
\bx^H \, := \, \prod_{(i,j) \. : \. H(i,j) \ts = \ts \boxplus} \. x_i \..
\end{equation}
In other words, \ts $\bx^H$ \ts is the product of \ts $x_i$'s
over all crosses \ts $(i,j)\in H$, see Figure~\ref{f:RC}.
Note that Schubert polynomials stabilize when fixed
points are added at the end, e.g.\ \ts $\Sch_{1432} = \Sch_{14325}$.  Thus we
can pass to the limit \ts $\Sch_w$\ts, where \ts $w\in S_\infty$ \ts is a
permutation \ts $\nn \to \nn$ \ts with finitely many nonfixed points.

Permutation $w\in S_n$ is called \defn{Grassmannian} if it has at most
one descent.  One can show that in this case $\Sch_w$ coincides with the
Schur function for the partition given by the \defn{Rothe diagram}
$$\rR(w) \, := \, \big\{\big(w(j),i\big) \. : \. i<j, \. w(i)>w(j)\big\} \. \ssu \ts \nn^2.
$$

It follows from this definition that the \defn{Schubert--Kostka numbers}
\ts $K_{w,\al} \ts := \ts [\bx^\al] \ts \Sch_w$ \ts
are nonnegative integers, and moreover that they are in $\SP$ as a counting function.
They are not known to be $\SP$-complete in the natural presentation
(this would imply Conjecture~\ref{conj:SC-unary} below).
See, however, \cite{ARY21} for other complexity properties of these
coefficients.

We refer to \cite{LLS21} for an alternative $\SP$ description of
the Schubert--Kostka numbers in terms of \emph{bumpless pipe dreams},
and to \cite{GH21} for a poly-time bijection relating two descriptions.
Finally, we refer to \cite{ST23} for the generalization of RC-graphs in
type $B$, $C$, and~$D$.  Note that all complexity properties of the
Schubert polynomials in type~$A$ (defined above), translate verbatim
to other types.

\medskip

\subsection{Schubert coefficients} \label{ss:intro-Schub-coeff}
It is well known and easy to see that Schubert polynomials \ts $\{\Sch_w \. : \. w\in S_\infty\}$ \ts
form a linear basis in the ring \ts $\zz[x_1,x_2,\ldots]$.
\defn{Schubert coefficients} \ts are defined as structure constants:
$$
\Sch_u \cdot \Sch_v \, = \, \sum_{w \ts \in \ts S_\infty} \. c^w_{u,v} \. \Sch_w\ts.
$$
It is known that \ts $c^w_{u,v} \in \nn$ \ts for all \ts $u,v,w\in S_\infty$\ts,
as they have both geometric and algebraic meanings \cite{BSY} which generalize
the number of intersection points of lines.  Since Schubert polynomials are Schur
polynomials for Grassmannian permutations, Schubert coefficients generalize the
LR-coefficients.

Denote by \ts $\SC$ \ts the function which computes Schubert coefficients
\ts $\{c^w_{u,v}\}$.  The following conjecture remains wide open:

\smallskip

\begin{conj}[{\rm \cite[$\S$13.4(10)]{Pak-OPAC}}{}]
\label{conj:SC-unary}
\ts $\SC$ \ts is \ts $\SP$-hard.
\end{conj}

\smallskip

This conjecture follow immediately from Conjecture~\ref{conj:LR-unary}.\footnote{The argument
in~\cite[p.~885]{MQ17} claiming that \ts $\{c^{w}_{uv}\}$ is $\SP$-hard via reduction
to the LR-coefficients is erroneous as it conflates different input sizes.}
It would follow from an even more basic conjecture that counting contingency tables
is $\SP$-complete when the input is in unary \cite[$\S$13.4{\small (1)}]{Pak-OPAC}.

A major open problem in Algebraic Combinatorics is whether Schubert coefficients
have a \defna{combinatorial interpretation} \cite[Problem~11]{Sta00}.
For various special cases, see \cite{Knu16, Kog01, MPPo} and a detailed discussion
in \cite[$\S$11.4]{Pak-OPAC}.  Following \cite{IP22}, this can be rephrased
as a problem whether \ts $\SC$ \ts is in~$\SP$.  We refer to \cite{KZ17,KZ23} for
the most general $\SP$-formulas for permutations with at most $3$ descents,
and for permutations with separated descents (cf.~$\S$\ref{ss:intro-prior}).

\smallskip

\begin{conj}[{\cite[Conj.~10.1]{Pak-OPAC}}{}]\label{conj:main-Schubert-SP}
\ts $\SC$ \ts is not in \ts $\SP$.
\end{conj}

\smallskip

The following argument by Alejandro Morales shows that \ts $\SC\in \GapP:=\SP-\SP$, see \cite[Prop.~10.2]{Pak-OPAC}.
Let \ts $\si\in S_n$ \ts and let \ts $\rho_n:=(n-1,\ldots,1,0)\in \nn^n$.
Define
$$\Phi(\si) \. := \. \big\{(\al,\be,\ga)\in (\nn^n)^3\.: \.  \al+\be+\ga=\si\rho_n\big\}.
$$
It was shown by Postnikov and Stanley in \cite[Cor.~17.13]{PS09}, that Schubert coefficients
have a \emph{signed combinatorial interpretation}:
\begin{equation}\label{eq:PS}
c^w_{u,v} \, = \, \sum_{\si\in S_n} \, \sum_{(\al,\be,\ga)\.\in\. \Phi(\si)} \,
\sign(\si) \. K_{u, \al} \. K_{v, \be}\. K_{w, \ga}\..
\end{equation}
Since the Schubert--Kostka numbers $\{K_{u,\al}\}$ are in $\SP$ by~\eqref{eq:Schubert-def},
this implies that Schubert coefficients are in $\GapP$.  An alternative proof
which generalizes to other root systems, is given by the authors in \cite[Thm~1.4]{PR24}.
\medskip

\subsection{Main results}\label{ss:intro-vanish}
The \defn{Schubert vanishing problem} \ts is a decision problem
$$
\SV \, := \, \big\{c^w_{u,v} =^? 0 \big\}.
$$
This problem is also heavily studied and resolved in special cases,
particularly in type~A.  We postpone an extensive discussion of the prior work
until~$\S$\ref{ss:intro-prior}.

We now proceed to state the main result of the paper.
Recall the complexity class \ts $\AM\subseteq \Pip_2$ \ts
 of decision problems that can be decided in polynomial time by an
 \emph{Arthur--Merlin protocol} with two messages, see e.g.\ \cite{AB09,Pap94}.
See~$\S$\ref{ss:CS-more} for connections to other complexity classes.

\smallskip

\begin{thm}[{\rm Main theorem}{}]
\label{t:main-AM}
\. $\SV$ \ts is in \ts $\coAM$ \ts assuming \ts $\GRH$. The result holds for the vanishing of
Schubert coefficients in types~$A$, $B$, $C$ and~$D$.
\end{thm}

\smallskip

Here the \ts $\GRH$ \ts stands for the \emph{Generalized Riemann Hypothesis}, that
all nontrivial zeros of $L$-functions \ts $L(s,\chi_k)$ \ts have real part~$\frac12$.
In fact, tracing back the references shows that a weaker assumption,
the \emph{Extended Riemann Hypothesis} ($\ERH$) suffices.  We stick with
the $\GRH$ as better known, and refer to \cite[$\S$6]{BCRW08} for definitions
and relationships between these hypotheses, and to \cite{Roj07} for discussion
of an even weaker number-theoretic assumption.

\smallskip

Main Theorem~\ref{t:main-AM} proves the result for all classical
types.\footnote{The first version of this paper claimed the result only for
types \ts $A$, \ts  $B$ \ts and~$\ts C$, but not for~$\ts D$, see \cite{PR24o}.
For non-classical types \ts $E_6$, \ts $E_7$,
\ts $E_8$, \ts $F_4$ \ts and \ts $G_2\ts$,  there is only a finite number of
Schubert coefficients, so the problem is computationally uninteresting.}
The theorem has several far-reaching complexity implications,
see~$\S$\ref{ss:disc-imply}.  Notably, the
theorem shows that \ts $\SV$ \ts is in the polynomial hierarchy $\PH$
(assuming $\GRH$).  This was out of reach until now.

Indeed, until recently, \ts
$\SV \subseteq\PSPACE$ \ts was the only known upper bound.
Morales's observation above gives \ts $\SV \in \CEP$.
By itself this does not suggest that $\SV \in \PH$.  Indeed,
Tarui's theorem implies that \ts $\CEP$ \ts is not in
\ts $\PH$ \ts unless \ts $\PH$ \ts collapses to a finite level:
\ts $\PH = \Sigmap_m$ for some~$m$, see~$\S$\ref{ss:CS-count}.

\smallskip

To contrast the Main Theorem, we make a conjecture in the opposite direction:

\smallskip

\begin{conj}
\label{conj:main-Schubert-vanishing}
\ts $\SV$ \ts is not in \ts $\coNP$.
\end{conj}

\smallskip

We have mixed feelings about this conjecture.  On the one hand,
Main Theorem~\ref{t:main-AM} suggests that Conjecture~\ref{conj:main-Schubert-vanishing}
is false, see~$\S$\ref{ss:disc-imply}{\small $(3)$}.  On the other hand, note that
if \ts $\SC \in \SP$, then \ts $\SV \in \coNP$.  Thus,
Conjecture~\ref{conj:main-Schubert-vanishing} implies Conjecture~\ref{conj:main-Schubert-SP}.

\begin{conj}[{\rm cf.~\cite[$\S$4]{ARY19}}{}]
\label{conj:main-Schubert-vanishing-coNP}
\ts $\SV$ \ts is \ts $\coNP$-hard.
\end{conj}

This is a decision counterpart of Conjecture~\ref{conj:SC-unary}.



\medskip

\subsection{Prior work on the Schubert vanishing problem} \label{ss:intro-prior}
The literature on the vanishing of Schubert coefficients and its various extensions
is too extensive to be fully reviewed.  Below are some highlights that
are most relevant from the complexity theoretic point of view.  Although many
of these conditions extend to all types, we restrict our presentation to type~A,
which is also best studied.  We refer to \cite[$\S$5]{StDY22} for a comparison
of some of these conditions from a combinatorial point of view.

\subsubsection{Poly-time conditions}\label{sss:special-poly-time}
It follows immediately from the algebraic definition via divided differences,
that \ts $c^w_{u,v}=0$ \ts when \ts $\des(w) > \des(u)+\des(v)$.
In \cite{Knutson01}, Knutson shows that \ts \ts $c^w_{u,v}=0$ \ts when \ts
$\Des(u) \cap \Des(v) \cap \Des(w\ts w_\circ)\ne\emp$ (see also
\cite[Cor.~4.15]{PW24}).  The apparent asymmetry
disappears in view of Equation~\eqref{eq:SchubStructure-sym} below.
In both cases, the sufficient conditions for Schubert vanishing can
be trivially verified in polynomial time.

Another simple sufficient condition for Schubert vanishing is given by
\ts $u \not\leqslant w$, where
\ts ``$\leqslant$'' is the \emph{strong Bruhat order}.  This condition
follows immediately from the algebraic definition (see e.g.\ \cite[$\S$5.1]{StDY22}).
Famously, verifying that \ts $u \not\leqslant w$ \ts can be done in polynomial time
via \emph{Ehresmann's tableau condition}, see e.g.\
\cite[Prop.~2.1.11]{Man01}.

In~\cite[$\S$1.2, \ts $\S$4.3]{StDY22}, St.~Dizier and Yong
gave a necessary condition for nonvanishing of Schubert coefficients
in terms of certain fillings of Rothe diagrams of
permutations $u,v,w$ with weights given by permutations satisfying
additional inequalities.  Using \cite{ARY19}, the authors then prove \cite[Thms.~A and~B]{StDY22},
that the existence of such tableaux can be decided in polynomial time, thus
giving polynomial time sufficient conditions for the vanishing of
Schubert coefficients.

Most recently, \cite[Cor.~5.12]{HW24} showed that \ts $\ze(w) > \ze(u)+\ze(v)$ \ts
implies that $c^w_{u,v}=0$, where \ts $\ze(w)$ \ts is the number of nonzero rows
in the Rothe diagram \ts $\rR(w)$.  This gives yet another simple sufficient
condition for vanishing of Schubert coefficients.

\smallskip

\subsubsection{$\NP$ conditions}\label{sss:special-NP}
Every time there is a combinatorial interpretation of Schubert
coefficients is given in a special case, this can be viewed as a
sufficient condition for non-vanishing.  Formally, a $\SP$ formula
for \ts $c^w_{u,v}$ \ts gives an $\NP$ sufficient condition for deciding
\ts $\big\{c^w_{u,v}>^?0\big\}$, since a single combinatorial object counted
by the formula gives a polynomial witness for positivity.

As we mentioned above, Knutson and Zinn-Justin \cite{KZ17} gave a
remarkable combinatorial interpretation for Schubert coefficients \ts $c^w_{u,v}$ \ts where \ts
$\des(u), \des(v), \des(w)\le 3$.  Similarly, in a followup paper \cite{KZ23}, the authors gave a
combinatorial interpretation for the case \ts $\Des(u) \cap \Des(v)=\emp$, and, more generally, for
\ts $\big|\Des(u) \cap \Des(v)\big|\leq 3$.
Both cases contain Grassmannian permutations as special case, and thus give a
far-reaching generalizations of LR-coefficients.
Since computing the set of descents is in~$\poly$, the Schubert vanishing problem
\ts $\big\{c^w_{u,v} =^? 0 \big\}$ \ts is in $\NP$ in these cases.

In \cite{Purbhoo06}, Purbhoo presents two sufficient conditions: one
for vanishing and one for non-vanishing of Schubert coefficients.  These
conditions are given combinatorially, in terms of what he calls \emph{root games}
which can be viewed as certain sequences of subsets of elements of the poset of
positive roots.  In type~$A$, the poset is isomorphic to the shifted staircase.
The sequence start at \emph{initial position} and end with a \emph{winning position},
with steps given by simple moves.

The sufficient condition for vanishing is given by a \emph{doomed position} where
no move can be made \cite[Thm~3.6]{Purbhoo06}, a property verifiable in
polynomial time  (cf.\ \cite[Thm~4.1.6]{Sea15}).
A sufficient condition for non-vanishing is given by a \emph{winning root game}
\cite[Thm~3.7]{Purbhoo06}.    Since each winning game can be verified in~$\poly$,
this sufficient condition is in~$\NP$.


Finally, in \cite[Thm~5.1]{BV08}, Billey and Vakil study the \emph{permutation arrays}
introduced by Eriksson and Linusson \cite{ELa,ELb}.  These permutation arrays
are subsets of \ts $\{1,\ldots,n\}^4$ \ts with certain poly-time checkable conditions.
The authors used elimination theory to
give a necessary condition for non-vanishing of Schubert coefficients
\ts $c^w_{u,v}>0$ \ts in terms of a unique permutation array with certain properties.
Existence of such permutation array can be verified in polynomial time, making
this necessary condition in~$\NP$ (uniqueness is a byproduct of the construction).
See also \cite[Prop.~9.7]{AB07} by Ardila and Billey, for an extension of this rule.

\smallskip

\subsubsection{Algebraic conditions}\label{sss:special-alg}
In classical types, a necessary and sufficient condition for vanishing
of Schubert coefficients was given by Purbhoo \cite{Purbhoo06} (see also
\cite{Belkale06,PurbhooThesis}).
These are linear algebraic conditions, which can be viewed as a polynomial
sized system of complex algebraic equations where some variables are
assumed to be generic.  An important feature of this approach is
the dual nature of the system, as it gives an algebraic certificate for
vanishing rather than non-vanishing.
%
We give a quick description of this system in Lemma~\ref{lem:Pur}
in the appendix, as we use it to give complexity analysis of \ts $\SV$ \ts
in  the (nonstandard) BSS model of computation.

A more direct description of an algebraic system is given by Billey and Vakil
in \cite[Thm~5.4]{BV08}, which has exactly \ts $c^w_{u,v}$ \ts
solutions for a generic values of some variables.  They also describe
the system of conditions for these variable being generic under assumption
that the set of solutions is $0$-dimensional \cite[Cor.~5.5]{BV08}.
The authors do not give a complexity analysis for this
system and our approach does not apply;
see~$\S$\ref{ss:finrem-BW} for further details and a complexity discussion.

\medskip
\subsection{Hilbert's Nullstellensatz}\label{ss:intro-HN}
Let $R=\cc{}[x_1,\dots,x_s]$ for some $s>0$.
\defng{Hilbert's weak Nullstellensatz} \ts
is a fundamental result in Algebra, which states that a polynomial system
\begin{equation}\label{eq:HN-system}
f_1 \. = \. \ldots \. =  \. f_m \. = \. 0  \quad \text{where} \quad f_i\in R
\end{equation}
has no solutions over $\cc$ \. \underline{if and only if} \.
there exist \ts $(g_1,\ldots,g_m)\in R^m$, such that
$$
\sum_{i=1}^m \. f_i \. g_i \, = \, 1\ts.
$$

Now let \ts $f_1,\ldots,f_m\in \zz[x_1,\dots,x_s]$.
The decision problem $\HN$ (\defn{Hilbert's Nullstellensatz}),
asks if the polynomial system \eqref{eq:HN-system}
has a solution over \ts $\mathbb{C}$.\footnote{By the Nullstellensatz,
this is equivalent to asking if there is a solution over $\overline{\qqq}$.
}
Here and everywhere below, the \defn{size} of the polynomial system
is the sum of bit-lengths of the coefficients in the polynomials~$f_i$.

Mayr and Meyer \cite{MM82} showed that \ts $\HN$ \ts is in \ts $\EXPSPACE$,
and is also \ts $\NP$-hard.
In a series of major algebraic results \cite{Kol88,KPS01,Jel05}, it was
shown that one can take $g_i$ of single exponential size.  This in turn
reduces \ts $\HN$ \ts to a single-exponential-sized system of linear equations,
and implies that \ts $\HN$ \ts is in \ts $\PSPACE$.  This is a major
unconditional result.  It was not known if \ts $\HN$ \ts is in the
polynomial hierarchy until Koiran's breakthrough result:


\smallskip

\begin{thm}[{\cite[Thm~2]{Koiran96}}{}]\label{t:main-HN}
    \. $\HN$ \ts is in \ts $\AM$ \ts assuming \ts $\GRH$.
\end{thm}

\smallskip

For the proof, Koiran's needs existence of primes in certain
intervals and with modular conditions, thus the $\GRH$ assumption.
For the proof of Theorem~\ref{t:main-AM}, we need the following
strengthening of Theorem~\ref{t:main-HN} to finite algebraic extensions.
Let
$$
f_1,\ldots,f_m\.\in\. \mathbb{Z}(y_1,\ldots,y_k)[x_1,\dots,x_s]\ts.
$$
The decision problem $\HNP$ (\defn{Parametric Hilbert's Nullstellensatz})
asks if the polynomial system \eqref{eq:HN-system} has a solution over
\ts $\overline{\mathbb{C}(y_1,\ldots,y_k)}$.
In a remarkable recent work, Ait El Manssour, Balaji, Nosan, Shirmohammadi, and Worrell
extended Theorem~\ref{t:main-HN} to~$\HNP:$

\smallskip

\begin{thm}[{\cite[Thm~1]{A+24}}{}]\label{t:main-HNP}
    \. $\HNP$ \ts is in \ts $\AM$ \ts assuming \ts $\GRH$.
\end{thm}

\smallskip

The proof of the authors is rather interesting as it substantially
simplified and ``algebraized'' Koiran's original approach in
\cite{Koiran96,Koiran97}, avoiding the use of semi-algebraic geometry.
The background behind \ts $\HNP$ \ts is also very interesting,
as are its computational aspects.  We refer the reader to a \cite{A+24}
for an extensive discussion of this and other related work.

\medskip

\subsection{Schubert vanishing via Hilbert's Nullstellensatz}\label{ss:intro-SV-proof}
We prove Theorem~\ref{t:main-AM} as an application of Theorem~\ref{t:main-HNP}.
More precisely, we prove the following:

\begin{lemma}[{\rm Main lemma}{}] \label{l:SV-HNP}
\. $\neg\SV$ \. reduces to \. $\HNP$ \. in types~$A$, $B$, $C$ and~$D$.
\end{lemma}

The lemma, combined with Theorem~\ref{t:main-HNP} immediately implies Theorem~\ref{t:main-AM}.
Let us emphasize that the vanishing of Schubert coefficients in type~$D$
stated in Theorem~\ref{t:main-AM}, but is not included in the main part
of the paper.  The reason is that the reduction in proof of the lemma is both
delicate and technical, and does not go through in type~$D$, see
Remark~\ref{rem:type-D}.

In type~$D$, we obtain the result in Lemma~\ref{l:SV-HNP-app}.
For the proof, we introduce a different approach in the Appendix~\ref{app:typeD}
(joint with David Speyer).  To clarify the strength of this new
approach, we present the construction uniformly for all classical types.

\medskip

\subsection{Proof outline}\label{ss:intro-sketch}
%

\medskip

The proof of the Lemma~\ref{l:SV-HNP} is somewhat technical and based on
prior constructions.
But at the end, we produce an explicit polynomial system \eqref{eq:HN-system}
with polynomially many parameters.

Our starting point is an equivalent algebro-geometric definition of Schubert coefficients,
arguably the closest to Schubert's original work.  We use \emph{Kleiman's transversality theorem}
\cite{Kleiman} to interpret Schubert coefficients \ts $\{c_{u,v}^w\}$ \ts as counting the number
of points in the intersection of corresponding generically translated Schubert varieties.
This is a well-known approach, both in Algebraic Geometry \cite{AF} and in
computational applications \cite{HS17}.

From there, the main issue is to formulate the transversality condition
by adding new variables (parameters) and explicit polynomial equations.  Following
the original approach in \cite{HS17} in type~$A$, we obtain and analyze the
so called \defng{lifted formulations} for each type.
These polynomial systems are relatively clean in type~$A$, see Section~\ref{sec:liftA}
and an example in~$\S$\ref{app:Schub}.  In Section~\ref{sec:liftBC}, we present
lifted formulations in types~$B$ and~$C$.

A lifted formulation in type~$D$ similar to that in other classical types
is given in Appendix~\ref{sec:liftD}.  Unfortunately, that formulation does
not prove Lemma~\ref{l:SV-HNP} in type~$D$.  This is because one of the
equations is an $n\times n$ determinant evaluation: \ts $\det(X)=1$, and
which has exponentially many terms.  Recall that systems in $\HN$ and $\HNP$
must have polynomially many equations, each with polynomially many terms;
See also~$\S$\ref{ss:finrem-Pur} and Remark~\ref{rem:type-D}.
Still, the systems constructed in Appendix~\ref{sec:liftD} turned out
to be useful for results in Appendix~\ref{app:BSS}.

In Appendix~\ref{app:typeD}, we prove Lemma~\ref{l:SV-HNP} in type~$D$
(see Lemma~\ref{l:SV-HNP-app}).  We present an alternative construction
of lifted formulations in all classical types, including type~$D$.
This section is joint with David Speyer.

Without going to details, lifted formulations define membership in a Schubert variety
as a system of bilinear equations. They are more efficient than the primal–dual
formulation in \cite{HHS16} (cf.\ \cite{L+21}).  It is a minor miracle that
one can avoid determinantal equations and obtain the desired
polynomial size systems of polynomials in the three classical types
as in the lemma.  Finally, we note that we are using the
conventions in \cite{AIJK} which are different from those in~\cite{HS17}.

\medskip

\section{Discussion and implications} \label{s:disc}

To clarify the importance of our results, let us recall three
somewhat related problems.

\subsection{Comparison to other problems} \label{ss:disc-comp}
To better explain how surprising the results are in this paper, let
us recall complexity results on related problems.

\medskip

\nin
{\small $(i)$} {\bf \defna{Characters of the symmetric group}.} 
Let \ts $\chi^\la$ \ts denote irreducible characters of the symmetric
group $S_n$, where \ts $\la\vdash n$ \ts is a partition of~$n$.
They play the same role that Schur functions play in the representation
theory of $\GL_n(\cc)$.  The character values \ts $\chi^\la(\mu)$, where
$\la,\mu\vdash n$, are integers.  As a function, $\chi$ \ts is \ts $\GapP$-complete
even when the input is in unary, see \cite{Hep94}.

In \cite{IPP24}, Ikenmeyer, Pak and Panova showed that squared characters \ts
$\big(\chi^\la(\mu)\big)^2$ \ts are not in \ts $\SP$ \ts unless \ts $\PH=\Sigmap_2\ts$.
This is the analogue of Conjecture~\ref{conj:main-Schubert-SP}.  This was done
by proving that the \defn{character vanishing problem} \ts $\{\chi^\la(\mu) =^? 0\}$ \ts is $\CEP$-complete.
By Tarui's theorem (see~$\S$\ref{ss:CS-count}), this shows that the character vanishing
problem is not in $\PH$, unless $\PH$ collapses to a finite level.

\medskip

\nin
{\small $(ii)$} {\bf \defna{Defect of Stanley's poset inequality}.} 
Let $P=(X,\prec)$ be a finite poset on $n$ elements, and let $x,z_1,\ldots,z_k\in X$ be fixed
distinct elements. Fix distinct integers \ts $c_1,\ldots,c_k \in \{1,\ldots,n\}$.
\defn{Linear extensions} of $P$ are order-preserving bijections \ts $f: X\to \{1,\ldots,n\}$.
They generalize \emph{standard Young tableaux} of both straight and skew shape.

Let $N(a)$  denote the number of linear extensions $f$ of $P$,
s.t.\ $f(z_i)=c_i$ and $f(x)=a$.
Famously, Stanley's poset inequality gives log-concavity of $\{N(a)\}$:
$$N(a)^2 \, \ge \, N(a+1) \cdot N(a-1).
$$
Denote \ts $\de(a):= N(a)^2 - N(a+1) \cdot N(a-1)$ \ts the defect of this inequality.
It follows from \cite{BW91} that $\de$ is $\GapP$-complete.

Similarly to the example above, for $k\ge 2$, it was shown by Chan and Pak \cite{CP-AF},
that \ts $\de$ \ts is not in \ts $\SP$ \ts unless \ts $\PH = \Sigmap_2$.  This was done
by proving that the \defn{Stanley Vanishing problem} \ts $\{\de(a) =^? 0\}$ \ts
is not in \ts $\PH$ \ts unless \ts $\PH$ \ts collapses to the finite level.
In particular, the authors showed that the Stanley Vanishing problem is not in
\ts $\coNP$ \ts unless \ts $\PH = \Sigmap_2$, similar to the claim in
Conjecture~\ref{conj:main-Schubert-vanishing}.  See also~\cite{CP-coinc,CP-SY} for a
similar approach to a other \emph{coincidence problems}.

Rather remarkably, it remains open if the Stanley Vanishing problem is $\CEP$-complete
for $k\ge 2$, as the proof bypasses this problem.
Curiously, and in the opposite direction, for \ts $k\in \{0,1\}$ \ts
it was shown in \cite{SvH-acta} and \cite{CP-AF}
that the Stanley vanishing problem is in~$\poly$.  \defng{Pak's Conjecture} states
that the defect $\de$ is not in $\SP$ in these cases as well \cite[Conj.~6.3]{Pak-OPAC}.
At the moment, there are no known tools to establish this result.

\medskip

\nin
{\small $(iii)$} {\bf \defna{Kronecker coefficients}.} 
Let \ts $\la,\mu,\nu\vdash n$ \ts be partitions of~$n$.
Denote by \ts $g(\la,\mu,\nu)$ \ts the
\defn{Kronecker coefficients} defined as follows:
$$
g(\la,\mu,\nu) \, := \, \<\chi^\la\chi^\mu,\chi^\nu\> \, = \, \frac{1}{n!} \.
\sum_{\si \in S_n} \. \chi^\la(\si) \. \chi^\mu(\si)\. \chi^\nu(\si)\..
$$
By definition, \ts $g(\la,\mu,\nu)\in \nn$.  Kronecker coefficients
are the structure constants in the ring of $S_n$ characters, and thus
play the same role that LR-coefficients play for $\GL_n(\cc)$ characters.

Kronecker coefficients are closely related to Quantum Computing and appear naturally in
Geometric Complexity Theory, see e.g.\ \cite{Aar16,Mul09}.  They can also
be defined for all types and are known to be hard to compute and to
analyze \cite{CDW12,PP17}.

Finding a combinatorial interpretation for Kronecker coefficients remains a major open problem
in Algebraic Combinatorics first posed by Murnaghan in 1938, see
\cite[Problem~10]{Sta00}.  It was conjectured by the first author
\cite[Conj.~9.1]{Pak-OPAC},
that \ts $g$ \ts is not in~$\SP$, a conjecture that remains wide open.

It was shown by Ikenmeyer, Mulmuley and Walter in \cite{IMW17},
that $g$ is $\SP$-hard, while the Kronecker Vanishing Problem \ts
$\{g(\la,\mu,\nu) =^? 0\}$ \ts is  $\coNP$-hard, even when the input
is in unary.  In a different direction, it was shown by
Ikenmeyer and Subramanian  \cite{IS23}, that $g$ is in $\SBQP$.
Similarly, the vanishing problem was shown in $\coQMA$ by Bravyi et al.\
\cite{B+24}.  As pointed out in \cite[Rem.~1.2]{IS23},
these results suggest that proving the Kronecker Vanishing problem is
$\CEP$-complete (as was done in \cite{IPP24} for characters),
would be difficult.  Indeed, the authors note that this would
imply \ts $\NQP \subseteq \QMA$, an open problem in Quantum Computing
\cite{KMY09}, cf.~$\S$\ref{ss:CS-more}.

\medskip

\subsection{Connections and implications} \label{ss:disc-imply}
There are several ways to think of results in the paper.

\medskip

\nin
{\small (0)} \ts  Most straightforwardly,
we showed that the vanishing of Schubert coefficients is in $\coAM$, i.e.\
relatively low in the polynomial hierarchy (Theorem~\ref{t:main-AM}).
If one assumes Conjecture~\ref{conj:main-Schubert-vanishing}, this
upper bound is close to optimal.

\medskip

\nin
{\small (1)} \ts
Looking at the complicated sign-reversing structure of \eqref{eq:PS}
and assuming that the Schubert--Kostka
numbers are $\SP$-complete (cf.\ Conjecture~\ref{conj:SC-unary}),
one would naturally assume positive and
negative terms are ``independent'' \ts $\SP$-functions,
cf.~$\S$\ref{ss:disc-comp}{\small $(i)$}.
This would imply that the Schubert vanishing problem is $\CEP$-complete,
and prove Conjecture~\ref{conj:main-Schubert-vanishing-coNP}.
By Tarui's theorem (see $\S$\ref{ss:CS-count}),
this would imply that \ts $\SV\notin \PH$ \ts unless \ts $\PH$
\ts collapses to a finite level.  Main Theorem~\ref{t:main-AM}
implies that this independence assumption is flawed and
that in fact \ts $\SV \in \PH$ (assuming $\GRH$).

%

\medskip

\nin
{\small (2)} \ts  Another way to think about our results is an effort to
show that Conjecture~\ref{conj:main-Schubert-SP} \emph{cannot be proved} by
showing that the Schubert vanishing problem is not in~$\PH$.   This would
be a natural approach to the problem, similar to that in {\small $(i)$}
and {\small $(ii)$} discussed above.  All of this is assuming both the $\GRH$
and a non-collapse of $\PH$, of course. 
This is a much stronger argument than the implication
of an open problem in quantum computing discussed at
the end of~$\S$\ref{ss:disc-comp}{\small $(iii)$}.

Note that the \emph{relativization argument} \ts employed in
\cite{IP22} to prove that some functions are not in $\SP$
is inapplicable in this setting.  Thus, Theorem~\ref{t:main-AM}
implies that there are currently no other tools to attack
Conjecture~\ref{conj:main-Schubert-SP}.  The same applies to Pak's
conjecture in~$\S$\ref{ss:disc-comp}{\small $(ii)$},
 that the defect of Stanley's inequality for $k=0$, is not in~$\SP$.

\medskip

\nin
{\small (3)} \ts
There is a plausible argument in favor of the \emph{derandomization assumption} \ts $\BPP=\P$,
see \cite{Wig23} partly motivated by \cite{IW97}.\footnote{The poll in~\cite{Gas19} reports that
about $98\%$ of ``experts'' and about $60\%$ of ``non-experts'' believe that $\P=\BPP$. Since
the problem has been open for over 60 years,
it is worth considering all possibilities.
}
Several versions of this assumption imply that \ts $\AM = \NP$ \ts \cite{KvM02,MV05},
and disprove Conjecture~\ref{conj:main-Schubert-vanishing}.\footnote{We explore this direction
in a short followup paper \cite{PR24b} written from a combinatorial point of view.}
As we mentioned above, this by itself does not disprove Conjecture~\ref{conj:main-Schubert-SP};
cf.\ discussion of Pak's conjecture in~$\S$\ref{ss:disc-comp}{\small $(ii)$}.
In the opposite direction, there are also some important indications against
derandomization assumptions, see e.g.\ discussion in~\cite[$\S$4.3]{SY09}.  Notably, it is unlikely
that \ts $\HN\in \NP$ \ts as discussed in \cite[$\S$6]{Koiran96} (preprint version).

%
%

\medskip

\nin
{\small (4)} \ts Although the underlying algebraic properties of Kronecker and Schubert
coefficients are unrelated, they both generalize LR--coefficients whose vanishing problem
is in~$\P$.  There are other connections between complexity classes
of best upper bounds $\QMA$ and $\AM$ for the vanishing problems of Kronecker and Schubert
coefficients, respectively (see Bravyi et  al.\ \cite{B+24} and our Theorem~\ref{t:main-AM}).
For example, if \ts $\BQP \subseteq \AM$ \ts then \ts $\PH = \Sigmap_2$ \cite{Aar10}.
Of course, \ts $\BQP\subseteq \QMA$ \ts are not believed to be in $\PH$, see e.g.\ \cite{RT22}.

\medskip

\nin
{\small (5)} \ts
Main Theorem~\ref{t:main-AM} also shows that \ts $\SV$ \ts is not $\NP$-complete
(assuming $\GRH$), as this would imply a collapse in the polynomial hierarchy to
the second level: \ts $\PH = \Sigmap_2$ \ts \cite[Thm~2.3]{BHZ87}.


\medskip

\section{Background in computational complexity}\label{s:CS}

We refer to \cite{AB09,Pap94} for
the definitions and standard results, and to \cite{Aar16,Wig19} for further
background.  For the BSS computational model and basic results on the corresponding
polynomial and counting hierarchies, see \cite{BCSS98,BC13}.  Below we give a
somewhat informal and incomplete reminder of standard computational complexity classes.


\subsection{Basic notions and the polynomial hierarchy}
\label{ss:CS-basis}
The notion of a \defn{combinatorial object} can be viewed as follows.
A \defn{word} is a binary sequence $x\in \{0,1\}^\ast$.
The \defn{size} of $x$ is the length~$|x|$.  In other words,
combinatorial objects of size~$N$ are encoded by words
of length~$N$, which in turn correspond to integers
\ts $0\le a <2^N$.

The \defn{language} as a subset of words: $L\subseteq\{0,1\}^*$.
For example, the language of words which encode all
Hamiltonian graphs can be defined as collection of pairs $(G, C)$
where $G=(V,E)$ is a simple graph, \ts $E\subseteq \binom{E}{2}$,
and $C\subseteq E$ is a Hamiltonian cycle in~$G$.  Language $L$ is called a
\defn{polynomial} \ts if the membership problem \ts $x\in^? L$ \ts
can be decided in time polynomial in the size of~$x$ (\defng{poly-time}).

Class $\P$ is a class of decision problems which can be decided in polynomial
time.  For example, deciding if a graph is connected is in~$\P$.
Class $\PSPACE$ is a class decision problems which can be decided in
polynomial space.

Fix a polynomial language~$L$.
Class $\NP$ is a class of problems of the type: given prefix~$y$,
does there exist suffix~$z$, s.t.\ $yz \in L$?  For example, given
a simple graph~$G=(V,E)$, the problem whether $G$ has a Hamiltonian
cycle is in~$\NP$.
Class $\coNP$ is a class of problems opposite to $\NP$, i.e.\
given prefix~$y$, is it true that for all~$z$, we have $yz \notin L$?
For example, deciding if graph~$G$ is non-Hamiltonian is in~$\coNP$.
To see this, rephrase the problem in the positive: decide if for every subset
of edges $C\subseteq E$, we have $C$ is \emph{not} a Hamiltonian cycle in~$G$.

Polynomial hierarchy $\PH$ is a union of classes $\Sigmap_m$ and $\Pip_m$ of
the form
$$\aligned
\exists x_1 \. \forall  x_2  \. \exists x_3 \. \ldots \.  \exists (\forall) \ts x_m \ : \ \,& x_1x_2x_3 \cdots x_m \ts \in^? \ts L, \quad \text{and}\\
\forall x_1 \. \exists  x_2 \. \forall x_3 \. \ldots  \. \forall (\exists) \ts x_m \ : \ \, & x_1x_2x_3 \cdots x_m \ts \in^? \ts L\ts,
\endaligned
$$
respectively.  It is important here that $m$ is fixed.  Note that $\Sigmap_1=\NP$,
$\Pip_1 = \coNP$, and $\PH \subseteq \PSPACE$.
It is known that \ts $\BPP \subseteq \Sigmap_2 \cap \Pip_2$ \ts and
thus very low in the polynomial hierarchy.

For a class \ts {\sf X}, we write \ts {\sf X}-complete \ts for the class of
\defn{complete problems} in \ts {\sf X}, i.e.\ problems {\sc A} $\in$ {\sf X} \ts such that every
problem {\sc B} $\in $ {\sf X} \ts reduces to~{\sc A}.  Not all classes are known to have
complete problems.  Famously, {\sc 3SAT} and {\sc Hamiltonicity} are $\NP$-complete.
For two classes {\sf X} and $\sf Y$, we write \ts $\text{\sf X}^{\text{\sf Y}}$ \ts for a
class of problems which can be solved by a solver in {\sf X} with an oracle~{\sf Y}.

\smallskip

\subsection{Counting problems} \label{ss:CS-count}
Fix a polynomial language~$L$.
Class $\SP$ is a class of problems of the type:  compute the number of
suffixes~$z$ with a given prefix~$y$, such that \ts $yz \in L$.  For example,
compute the number of Hamiltonian cycles $C$ in a given simple graph~$G$.
Class $\PP$ is the corresponding class of decision problems,
whether the number above is greater than a given integer~$k$.

Class $\FP$ is a subclass of $\SP$ of counting problems which can be solved
in poly-time.  Class $\FPSPACE$ is a
class of counting problem which can be solved in polynomial space.
Clearly, \. $\FP \ts \subseteq \ts \SP  \ts \subseteq \ts \FPSPACE\ts$.
Counting hierarchy \ts $\CH$ \ts is a union of
classes of decision problems corresponding to classes $\Sigmap$,
whether the number of solutions is~$\ge k$.  Clearly, \ts
$\PP \ts \subseteq \ts \CH\ts \subseteq \ts \PSPACE$.

A \ts \defn{$\SP$ function} \ts $f: \{0,1\}^\ast \to \nn$ \ts
is a counting function corresponding to a problem in~$\SP$.
We write \ts $f\in \SP$ \ts in this case.
Let \. $\GapP: = \SP - \SP$ \. denote set of differences of
such functions, and let \. $\GapPP$ \. denote the set of
nonnegative functions in \ts $\GapP$.  For example,
Schubert coefficients, Kronecker coefficients and the
defect of Stanley's inequality (see~$\S$\ref{ss:disc-comp}),
are all in \ts $\GapPP\ts$.

\defn{Toda's theorem}, see e.g.\ \cite{AB09,Gol08,Pap94},
says that \ts $\PH \subseteq \P^{\SP}$.
%
It implies that a $\SP$-complete
function is not in~$\PH$ unless \ts $\PH = \Sigmap_m$ \ts for some~$m$,
i.e.\ the polynomial hierarchy \defng{collapses}.
\defn{Tarui's theorem} \cite{Tar91} states that \ts $\PH \subseteq \NP^{\CEP}$.
The proof is an easy argument based on Toda's theorem, see e.g.\ \cite{IP22}.

Complexity class \ts $\CEP$ \ts
is the class of decision problems whether two $\SP$-functions have equal values
on two given inputs.  Clearly, \ts
$\coNP \subseteq \CEP \subseteq \PSPACE$ \ts where the first inclusion follows
since one can take \ts $g\equiv 0$.
For a function \ts $f\in \SP$, the corresponding \defn{coincidence problem} \ts
is defined as
$$
\text{\sc C}_f \, := \, \big\{f(x)=^? f(y) \ : \ x, \ts y\in \{0,1\}^\ast\big\}\ts.
$$
Clearly, \. $\text{\sc C}_f \ts \in \ts \CEP$.
It follows from Tarui's theorem, that if \ts $\text{\sc C}_f$ \ts
is \ts $\CEP$-complete, then we have \ts $\big(f(x)-f(y)\big)^2\notin \SP$ \ts
unless \ts $\PH=\Sigmap_m$ \ts for some~$m$, see \cite{CP-coinc,IP22}.

\smallskip

\subsection{More involved complexity classes}
\label{ss:CS-more}

Class $\BPP$ is defined as a class of decision problems which can
be decided in \defn{probabilistic polynomial time}.  In other words,
in poly-time the problem \ts $\{x\in^? L\}$ \ts can be decided
with probability of error at most $\frac{1}{3}$.  This probability
can be amplified by a repeated application of the algorithm if necessary.

If $\BPP$ is a probabilistic version of $\P$, then $\EBPP$ is defined
as probabilistic version of $\NP$.  In this case the decision problem
asks if there is a word in the \ts \defn{$\BPP$ \ts language}, i.e.\ a
language whose membership problem \ts $\{x\in^? L\}$ \ts is in $\BPP$.

One can view complexity problems as \defng{interactive proofs}, with
quantifiers describing communications between prover and verifier.
When we have a $\BPP$ verifier, i.e.\ when the verifier (Arthur) has powers to
flip coins and the prover (Merlin) has unlimited computational powers.
such communication is called the \defn{Arthur--Merlin protocol}.
The number of quantifiers becomes the number of messages between the
prover and the verifier.

The complexity class \ts $\AM$ \ts is a class of decision problems that
can be decided in polynomial time by an Arthur--Merlin protocol with
two messages, see e.g.\ \cite[$\S$8.2]{AB09}.  Recall the inclusions
$$
\NP \ \subseteq \ \EBPP  \ \subseteq \ \MA  \ \subseteq \ \AM  \
\subseteq \ \Pip_2 \ \subseteq \ \PH\ts.
$$
Here $\MA$ is a standard but slightly different Arthur--Merlin
protocol with two messages that we will not need.
%
%
Famously, \defng{graph isomorphism} is in $\coAM$, since graph non-isomorphism
can be established by a simple interactive protocol (see e.g.\ \cite[Thm~8.13]{AB09}).

Finally, in Quantum Computing, one considers \emph{quantum Turing machines}.
Class $\BQP$ is the class of decision problems which can be solved in
polynomial time on such machines with probability of error at most \ts $\frac13\ts$.
This class should be viewed as analogous to~$\BPP$.  Similarly,
class $\QMA$ is defined as analogous to class~$\MA$.   Finally, class $\NQP$
is a quantum analogue of $\NP$, and it is known that \ts $\NQP=\coCEP$ \ts \cite{FGHP98}.

%
%


\medskip

\section{Definitions, notation, and geometric setup} \label{s:setup}

\subsection{Standard notation} \label{ss:setup-notation}
We use \ts $\nn=\{0,1,2,\ldots\}
$ \ts and \ts $[n]=\{1,\ldots,n\}$.  We use $[a,b]$ to denote the interval \ts
$\{a,a+1,\ldots,b\}\ssu \zz$.
From this point on, unless stated otherwise, the underlying
field is always~$\cc$.

We use \ts $e_1,\ldots,e_n$ \ts to denote the standard
basis in~$\cc^n$, and $\zero$ to denote zero vector.
We use bold symbols such as \ts $\bx$, $\by$, $\bal$, $\bbe$ \ts
to denote sets and vectors of variables, and bars such as \ts $\ov x$ \ts and \ts $\ov y$,
to denote complex vectors.
We say that \ts $A=(a_{ij})$ \ts is \emph{symmetric} \ts if
\ts $a_{ij}=a_{ji}$ \ts for all \ts $i,j\in [n]$.
Similarly, we say that \ts $A=(a_{ij})$ \ts is \emph{skew-symmetric} \ts if
\ts $a_{ij}=-a_{ji}$ \ts for all \ts $i,j\in [n]$.

Recall the following standard notation for \emph{complex reductive simple Lie groups}.
We have the \emph{general linear group} ${\sf GL}_n(\cc)$, the \emph{odd special orthogonal group}
${\sf SO}_{2n+1}(\cc)$, the \emph{symplectic group} ${\sf Sp}_{2n}(\cc)$ and
the \emph{even special orthogonal group}  ${\sf SO}_{2n}(\cc)$.  These groups
correspond to \emph{root systems} $A_n$, $B_n$, $C_n$ and $D_n$, respectively,
and are called  \emph{groups of type}~$A$, $B$, $C$ and~$D$, respectively.

To distinguish the types, we use parentheses in Schubert coefficients, e.g.\
\ts $c^w_{u,v}(A)$, and superscript in other cases, e.g.\ $X_u^A$.  We omit
the dependence on the type when it is clear from the context.  We use bullets
(both in subscript and superscript) to denote flags, e.g.\ \ts $F_\bu$ \ts and \ts $E^\bu$.

\medskip

\subsection{Geometric and Algebraic Combinatorics notation}\label{ss:setup-Schubert}
Throughout the paper we follow conventions of \cite{AF,AIJK}.  We assume
the reader is familiar with some basic notions and results in Algebraic Geometry,
Lie theory and specifically Schubert theory.  Below we recall some key notation
that we use throughout the paper.

Let \ts ${\sf G}$ \ts be a \emph{complex reductive Lie group}.
Take \ts ${\sf B}\subset {\sf G}$ and \ts ${\sf B}_{-}\subset {\sf G}$ \ts
to be the \defn{Borel subgroup}
and \defn{opposite Borel subgroup}, respectively.
The \defn{torus subgroup} \ts is defined as \ts ${\sf T}={\sf B}\cap {\sf B}_{-}\ts$.
The \defn{Weyl group} \ts is defined as
the normalizer \ts $\mathcal{W}\cong N_{\sf G}({\sf T})/{\sf T}$.
Finally, the celebrated \defng{Bruhat decomposition} \ts says that
$$
{\sf G} \ = \ \bigsqcup_{w\in \mathcal{W}} \. {\sf B}_{-} \ts \dot{w} \ts {\sf B}\ts,
$$
where \ts $\dot{w}$ \ts is the preimage of \ts $w$ \ts in the normalizer
\ts $N_{\sf G}(\sf T)$.

The \defn{generalized flag variety} \ts is defined as \ts  ${\sf G}/{\sf B}$.
Recall that ${\sf G}/{\sf B}$ \ts  has finitely many orbits under the left action of ${\sf B}_{-}\ts$.
These are called \defn{Schubert cells} \ts and denoted \ts $\Omega_w\ts$.
Schubert cells are indexed by the Weyl group elements \ts $w\in\mathcal{W}$.
See \cite{MT11} for additional background.

 The \defn{Schubert varieties} $X_w$ are the Zariski closures of Schubert cells~$\Omega_w\ts$.
The \defn{Schubert classes} $\{\sigma_w\}_{w\in \mathcal{W}}$ are the
Poincar\'e duals of Schubert varieties. These form a ${\mathbb Z}$--linear basis
of the cohomology ring \ts $H^{*}({\sf G}/{\sf B})$.
%
%
%
%
%
The \defn{Schubert coefficients} $c_{u,v}^w$ are defined as structure constants:
\begin{equation}\label{eq:SchubStructureCoh}
    \sigma_u \smallsmile \sigma_v \, = \, \sum_{w\in {\mathcal{W}}} \. c_{u,v}^{w}\sigma_{w}\..
\end{equation}
The fact that in type~$A$ these constants coincide with the definition in~$\S$\ref{ss:intro-Schub-coeff},
is proved in~\cite{LS82}.  This definition easily implies the $S_3$-symmetries of Schubert coefficients:
\begin{equation}\label{eq:SchubStructure-sym}
c_{u,v}^{w_\circ w}  \,  = \, c_{v,u}^{w_\circ w} \, = \, c_{u,w}^{w_\circ v }\,.
\end{equation}

By the \defng{Kleiman transversality} \cite{Kleiman}, coefficients \ts $c_{u,v}^w$ \ts
count the number of points in the intersection of generically translated Schubert varieties:
\begin{equation}\label{eq:SchubStructureInt}
    c_{u,v}^w \,  = \,\#\big\{X_u(E_{\bullet}) \cap X_v(F_{\bullet}) \cap X_{w_\circ w}(G_{\bullet})\big\},
\end{equation}
where \ts $E_{\bullet}\ts$, \. $F_{\bullet}$ \ts and \ts $G_{\bullet}$ are generic flags.
Here
 \ts $w_\circ$ \ts is the \emph{long word} \ts in~$\ts\mathcal{W}$,
see \cite[Section~15.1]{AF} for details.

\medskip

\section{Lifted formulation in type $A$}\label{sec:liftA}

\subsection{Definitions} \label{ss:type-A-def}
In type $A$, we have \ts ${\sf G}:={\sf GL}_n(\cc)$. The Weyl group
\ts ${\mathcal W}=S_n$ \ts is the symmetric group on $n$ letters.
We also have \ts ${\sf G}/{\sf B}={\mathcal{F}l_n}$ \ts is
the \defn{complete flag variety}.
Here \ts ${\mathcal{F}l_n}$ \ts consists of flags \.
$$F_{\bullet}\, := \, \zero \subset F_1\subset \dots \subset F_n \. = \. \mathbb{C}^n
\quad \text{where} \quad \dim F_i=i\ts.
$$
Fix the \defn{coordinate flag} \ts $E^{\bullet}$ \ts given by \ts $E^i:=\langle e_n,\ldots,e_{i+1}\rangle$.
  We explicitly define Schubert cells as
\[
\Omega_w^{A}(E^{\bullet}) \, = \, \big\{F_{\bullet} \ : \ \dim(F_i\cap E^j)= k_w(i,j) \text { for } i,j\in[n]\big\},
\]
where
$$
k_w(i,j)\, := \,\#\big\{r\in[i] \. : \. w_r>j\big\} \quad \text{for all} \quad 1\le i,j\le n\ts.
$$
Then the Schubert variety in type $A$ is an orbit closure which can also be defined as
$$
X_w^{A}(E^{\bullet}) \. = \. \ol{\Omega_w^{A}}(E^{\bullet}).
$$

Take \ts $F^{\bullet}:=\pi E^{\bullet}$ \ts and \ts $G^{\bullet}:=\rho E^{\bullet}$,
where \ts $\pi=(y_{ij})$ \ts and \ts $\rho=(z_{ij})$ \ts are matrices of indeterminates.
Let \ts $\by=\{y_{11},\dots,y_{nn}\}$ \ts and \ts $\bz=\{z_{11},\dots,z_{nn}\}$ \ts be
the sets of variables in these matrices. Fix three permutations \ts $u,v,w\in S_n\ts$.

We use the system for solving
$$
X_{{u}}^{A}(F^{\bullet}) \. \cap \. X_{{v}}^{A}(G^{\bullet}) \. \cap \. X_{{w}}^{A}(E^{\bullet})
$$
as formulated in  \cite[Rem.~2.7]{HS17}. Since our conventions are opposite those in \cite{HS17},
we include the necessary translations here.

Using the classical determinantal formulation of the Schubert problem or an abridged version as in Billey--Vakil \cite{BV08}, one can similarly define a system of equations to solve the Schubert problem. However, these may involve exponentially many equations. See \cite{HS17} for further discussion as well as comparisons with other formulations of Schubert systems.

Let \ts $d:= \max \Des(w)$ \ts denote the maximal descent of~$w$.
The \defn{Stiefel coordinates} $\mathcal{X}_{{w}}^{A}$  for the Schubert cell $\Omega_{{w}}^{A}$ are the collection of $n\times d$ matrices $(m_{ij})$ such that
\begin{align}
m_{ij}=
    \begin{cases}
        1  & \text{ if } i={{w}}_j\\
        0  & \text{ if } i<{w}_j \text{ or } {w}^{-1}_i<j \\
        x_{ij}  & \text{ otherwise.}\\
    \end{cases}
\end{align}
For convenience, we relabel matrix entries
$\{x_{ij}\}$ as $\bx=\{x_1,\ldots, x_{s}\}$, where \ts $s= \binom{n}{2}-\inv(w)$,
reading across rows, top to bottom.
To every \ts $\omega\in \mathcal{X}_{{w}}^{A}$ with entries ${\bx}$, we associate the flag given by
the span of the column vectors $e_i(\bx)$ of $\omega$.
Take
\begin{align*}
    \bal \, &:= \, \big\{\alpha_{ij}: j\in[i], \text{ where } {u}_j<{u}_i\big\},\\
    \bbe \, &:=\, \big\{\beta_{ij}: j\in[i], \text{ where } {v}_j<{v}_i\big\}.
\end{align*}
For $i\in[d]$, define the $n$-vectors \ts $g_i(\bx,\bal)$ \ts and \ts $h_i(\bx,\bbe)$ \ts with entries in  \ts $\mathbb{Z}[\bx,\bal,\bbe]$  \ts such that
\begin{align*}
    g_i(\bx,\bal)&\, := \, e_i(\bx)\. + \. \sum_{\substack{j \in[i] \\ {u}_j<{u}_i}}\. \alpha_{ij \. }e_j(\bx)\ts,\\
    h_i(\bx,\bbe)&\, := \, e_i(\bx)\. + \. \sum_{\substack{j \in[i] \\ {v}_j<{v}_i}}\. \beta_{ij} \. e_j(\bx)\ts.
\end{align*}

\smallskip

\subsection{Characterization of Schubert varieties} \label{ss:type-A-tech}
In this section we outline a key construction of \cite{HS17}. Using the following fact, one may introduce auxiliary variables to characterize membership in the Schubert cell in terms of bilinear equations rather than determinants.

\smallskip


\smallskip

\begin{prop}[{\cite[Lemma~2.1]{HS17}}{}]\label{prop:liftInt}
    Let \ts $\Phi_{\bullet}\in {\mathcal{F}l_n}$ \ts be in general position with \ts $E^{\bullet}$,
    where \ts  $\Phi_i=\langle f_1,\ldots, f_i\rangle$ \ts for all \ts $i\in[d]$.
    Take \ts ${w}\in S_n$ \ts with \ts $d = \max\Des(w)$.
    Then \. $\Phi_{\bullet}\in \Omega_{w}^{A}(E^{\bullet})$ \. \underline{if and only if} \.
    for each \ts $k\in [d]$, there exist unique \ts $\{\beta_{jk}\}$ \ts such that
    \[
    g_k \, = \, f_k \. + \. \sum_{\substack{j \in[k] \\ {w}_j<{w}_k}}\. \beta_{jk} \. f_j \, \in \, E^{{w}_k-1}\. - \. E^{{w}_k}\..
    \]
\end{prop}

Note that the statement of the proposition is translated into our conventions
to make it amenable to generalization for other types.  We include their proof
for completeness.

\begin{proof}
    Suppose $\Phi_{\bullet}\in \Omega_w^{A}(E^{\bullet})$. We proceed by induction on $k$. For the base case, this translates to $g_1=f_1\in E^{w_1-1}-E^{w_1}$, which follows since $\Phi_{\bullet}\in \Omega_w^{A}(E^{\bullet})$. Suppose the result holds for all \ts $k'<k$.
Note that \. $\Phi_{k-1}+(E^{w_k-1}\cap \Phi_k)$ \. must be $0$ or $1$-dimensional modulo \ts $\Phi_{k-1}$.
    Thus there must exist some $f\in \Phi_{k-1}$ such that $f_k+f\in E^{w_k-1}$. Since $\Phi_{\bullet}\in \Omega_w^{A}(E^{\bullet})$, we further have $f_k+f\in E^{w_k-1}-E^{w_k}$.

    Since \ts $f\in \Phi_{k-1}$ \ts and \ts $f=\sum_{j<k}\alpha_j f_j\ts$,
    by the inductive assumption we have \ts $f=\sum_{j<k}\alpha_j' g_j\ts$. Note that if \ts $w_j>w_k$ \ts and \ts $E^{w_j-1}\subset E^{w_k-1}$,
    we may assume that
    $$f \, = \, \sum_{\substack{j<k \\ w_j<w_k}}\. \alpha_j' \. g_j\ts.
    $$
    Since \ts $g_j\in E^{w_j-1}-E^{w_j}$\ts, we obtain uniqueness of \ts $\alpha_j'\ts$.
    Using the inductive assumption, we can now rewrite each such $g_j$ in terms of $f_i$'s, giving
    \[g_k \, := \, f_k \ts + \ts f \, = \, f_k\. + \. \sum_{\substack{j \in[k] \\ w_j<w_k}}\. \beta_{jk}\. f_j \, \in \, E^{w_k-1}\ts - \ts E^{w_k}\..
    \]
    This completes the proof.
\end{proof}

\smallskip

\subsection{The construction} \label{ss:type-A-construction}
Let $d$ denote the maximal descent among ${u},{v},w:$
$$
d \. := \. \max  \big( \Des(u) \cup \Des(v) \cup \Des(w)\ts\big)\ts.
$$
We consider \ts $\omega\in \mathcal{X}_{{w}}^{A}$ with  column vectors $e_j(\bx)$. By construction, the corresponding flag $\omega^{\bullet}$ lies in $\Omega_w^{A}(E^{\bullet})$. Then construct $g_i(\bx,\bal)$ and $h_i(\bx,\bbe)$ in terms of $e_j(\bx)$.

By Proposition~\ref{prop:liftInt}, $\omega^{\bullet}\in \Omega_u^{A}(F^{\bullet})$ if and only if for every \ts $i\in[d]$, $g_i(\bx,\bal)\in F^{{u}_i-1}-F^{{u}_i}$ for $\bal$ unique.
To impose this condition,
we require
\begin{align}\label{eq:mainEq1A}
    (y_{j1} \hspace*{0.2cm} y_{j2}\hspace*{0.2cm}\cdots\hspace*{0.2cm} y_{jn}) \. \cdot \. g_i(\bx,\bal) \, = \, 0\ts, \text{ for each \ts $j<{u}_i$\ts.}
\end{align}

Similarly, $\omega^{\bullet}\in \Omega_v^{A}(G^{\bullet})$ if and only if for every \ts $i\in[d]$, $h_i(\bx,\bbe)\in G^{{v}_i-1}-G^{{v}_i}$ for $\bbe$ unique.
For this last condition, we need
\begin{align}\label{eq:mainEq2A}
  (z_{j1} \hspace*{0.2cm} z_{j2}\hspace*{0.2cm}\cdots\hspace*{0.2cm} z_{jn})\. \cdot \.  h_i(\bx,\bbe) \, = \, 0 \ts \text{, for each \ts $j<{v}_i$\ts.}
\end{align}

Let \. $\mathcal{S}^{A}(u,v,w)$ \. denote the system given by Equations~\eqref{eq:mainEq1A} and~\eqref{eq:mainEq2A}.
Note that the system \. $\mathcal{S}^{A}(u,v,w)$ \. consists of at most \ts $3\binom{n}{2}=O(n^2)$ \ts
bilinear equations and variables with coefficients in \ts $\mathbb{Z}[\by,\bz]$.
Moreover, all these coefficients have monomials in \ts $\{0,1\}$.

Define $w_\circ:=(n,n-1,\ldots,1)\in S_n$ to be the long word.
Let \ts $c_{u,v}^w(A)$ \ts
denote the Schubert coefficients in type~$A$ defined as in~$\S$\ref{ss:intro-Schub-coeff} or by
\eqref{eq:SchubStructureCoh}.

\smallskip

\begin{prop}\label{prop:schubHNredA}
  The system \. $\mathcal{S}^{A}(u,v,w_\circ w)$ \. has \. $c_{u,v}^w(A)$ \. solutions over \ts $\overline{\mathbb{C}(\by,\bz)}$.
\end{prop}
\begin{proof}
We first note
\begin{align*}
    c_{u,v}^w(A) \, &= \, \#\big\{X_{{u}}^{A}(F^{\bullet}) \cap X_{{v}}^{A}(G^{\bullet}) \cap X_{w_\circ w}^{A}(E^{\bullet})\big\}.\\
    &= \#\big\{\Omega_u(F^{\bullet}) \cap \Omega_v(G^{\bullet}) \cap \Omega_{w_\circ w}(E^{\bullet})\big\}.
\end{align*}
The first equality follows by definition. The next follows since by Kleiman transversality \cite{Kleiman}, all points in the first intersection of Schubert varieties must lie in the interior.

By Proposition~\ref{prop:liftInt},
for generic choices of $\by,\bz$ solutions to $\mathcal{S}^{A}(u,v,w_\circ w)$ are biject with flags $\omega^{\bullet}\in \Omega_u(F^{\bullet}) \cap \Omega_v(G^{\bullet}) \cap \Omega_{w}(E^{\bullet})$. Thus for generic evaluations \ts $\ov y, \ov z$ \ts of $\by, \bz$, respectively, the system  \. $\mathcal{S}^{A}(u,v,w_\circ w)$ \. has \ts $c_{u,v}^w(A)$ \ts solutions over \ts $\mathbb{C}$.

    Take \.
    $I\subseteq \mathbb{Z}(\ov y,\ov z)[\bx,\bal, \bbe]$ \. to be the ideal generated by the vanishing expressions in the system \ts $\mathcal{S}^{A}(u,v,w_\circ w)$.\footnote{Here \ts $I\subseteq \kk[\bx,\bal, \bbe]$ \ts where $\kk=\zz[\ov y,\ov z]$.
    We write \ts $\mathbb{Z}(\ov y,\ov z)[\bx,\bal, \bbe]$ \ts to make presentation uniform in all types and clarify the
    connection to $\HNP$ as presented in \cite{A+24}.
    }
A generic choice of evaluation has the ideal~$I$ be zero--dimensional and of degree \ts $c_{u,v}^w(A)$.
Since variables $\by,\bz$ are independent, their evaluations $\ov y$ and $\ov z$ are algebraically
independent, and the result follows.
\end{proof}

\smallskip

\subsection{An example}\label{app:Schub}
For the purposes of illustration, we formulate the system of equations
for the vanishing problem in the case $n=4$.

\smallskip

Let $E^{\bullet}$ be the standard flag.
Take $F^{\bullet}:=\pi E^{\bullet}$ and $G^{\bullet}:=\rho E^{\bullet}$ for $\pi =(y_{ij})$ and $\rho=(z_{ij})$ matrices of indeterminates.
Take $u=2143$, $v=3124$, and $w=4132$ in $S_4$. Then $w_\circ w=1423$. Note $d=3$ in this case.

Then the Stiefel coordinates are given by
\begin{equation*}
\mathcal{X}_{w_\circ w}^A \ = \ \left\{
{\small    \begin{pmatrix}
1 & 0 & 0 \\
x_1 & 0 & 1 \\
x_2 & 0 & x_4 \\
x_3 & 1 & 0
\end{pmatrix}}
: x_1,x_2,x_3,x_4\in \mathbb{C}
\right\}.
\end{equation*}
Let $e_i(\bx)$ denote the $i$-th column of \ts $\mathcal{X}_{w_\circ w}^A$ \ts for \ts $i\in \{1,2,3\}$.
To restrict solutions to the intersection \ts $X_{w_\circ w}^A(E^{\bullet})\cap X_{u}^A(F^{\bullet})$, we require
\begin{align*}
    g_1(\bx,\bal) \, &= \, c_1(\bx)\in F^{1}-F^{2},\\
    g_2(\bx,\bal)\, &= \, c_2(\bx)\in F^{0}-F^{1}, \text{ and}\\
    g_3(\bx,\bal)\, &= \, \al_{3,1}c_1(\bx)+\al_{3,2}c_2(\bx)+c_3(\bx)\in F^{3}-F^{4}.
\end{align*}
To restrict solutions to the intersection \ts $X_{w_\circ w}^A(E^{\bullet})\cap X_{{u}}^A(F^{\bullet})\cap X_{{v}}^A(G^{\bullet})$, we require
\begin{align*}
    h_1(\bx,\bbe)  \, &= \, c_1(\bx)\in G^{2}-G^{3},\\
    h_2(\bx,\bbe) \, &= \, c_2(\bx)\in G^{0}-G^{1}, \text{ and}\\
    h_3(\bx,\bbe) \, &= \, \be_{3,2}c_2(\bx)+c_3(\bx)\in G^{1}-G^{2}.
\end{align*}
These inclusions determine the following system \. $\mathcal{S}^{A}(u,v,w_\circ w):$
$$
\left\{
\aligned
\ & y_{k1}(\al_{3,1})+y_{k2}(\al_{3,1}x_1+1)+y_{k3}(\al_{3,1}x_2+x_4)+y_{k4}(\al_{3,1}x_3+\al_{3,2}) \ = \ 0 \ \text{ for } \ k\in \{1,2,3\},\\
\ & y_{11}+y_{12}(x_1)+y_{13}(x_2)+y_{14}(x_3) \ = \ 0 ,\\
\ & z_{k1}+z_{k2}(x_1)+z_{k3}(x_2)+z_{k4}(x_3) \ = \ 0 \ \text{ for } \ k\in \{1,2\},\\
\ & z_{12}+z_{13}(x_4)+z_{14}(\be_{3,2}) \ = \ 0.
\\
\endaligned \right.
$$
Note that this is a square system in \ts $\kk[x_1,x_2,x_3,x_4,\al_{31},\al_{32},\be_{32}]$, with
$7$ equations in~$7$ variables, and coefficients in \ts $\kk = \zz[y_{11},y_{12},\ldots,y_{34},z_{11},z_{12},\ldots,z_{24}]$.

\medskip

\section{Lifted formulations in types $B$ and $C$}\label{sec:liftBC}

\subsection{Definitions} \label{ss:type-B-def}
Let \. $\{e_{\ol{n}}, \ts \ldots \ts , \ts e_{\ol{1}}, \ts e_1, \ts \ldots \ts , e_n\}$ \. denote a basis in \ts $\mathbb{C}^{2n}$.
Define the \defn{skew symmetric} bilinear form on \ts $\mathbb{C}^{2n}$, such that \.
$\langle e_i, e_j\rangle=\langle e_{\ol{i}}, e_{\ol{j}}\rangle=0$ \. and \. $-\langle e_i, e_{\ol{j}}\rangle=\langle e_{\ol{j}}\ts, \. e_{i}\rangle=\delta_{ij}$ \ts for all \ts $i,j\in[n]$. Here and throughout this section we use the notation \ts $\ol{k}:=-k$ \ts
for all \ts $k\in[n]$.

In type~$C$, we have \ts ${\sf G}={\sf Sp}_{2n}(\cc)$ \ts the symplectic group,
defined with respect to the bilinear form \ts $\langle \cdot,\cdot\rangle$.
The corresponding Weyl group is the \defn{hyperoctahedral group} \.
$\mathcal{W}_n=S_n\ltimes \{\pm 1\}^n$, which can be viewed as
the set of \defn{signed permutations} of $n$ letters.
It will be convenient to realize $\mathcal{W}_n$ as a subgroup of
$S_{2n}$ where we impose \ts $w_{\ol{i}}=\ol{w_i}$ \ts for each \ts $i\in[n]$.

An \defn{isotropic flag} \ts $F_\bullet$ \ts is a flag of spaces
\[
\zero \subset F_n\subset F_{n-1}\subset \dots\subset F_{1}\subset \mathbb{C}^{2n},
\]
where \ts $\dim F_i=n+1-i$ for each $i\in[n]$ and each $F_i$ is \defn{isotropic},
i.e.\ satisfies \ts $F_i\subseteq F_i^{\perp}$.
We identify the flag variety \ts ${\sf G}/{\sf B}$ \ts
with the set of such isotropic flags.

Recall that \emph{torus subgroup} \ts ${\sf T}$ \ts in this case consists
of diagonal matrices $T$ given by
\ts ${\sf diag}(T)=(a_1,\ldots,a_{2n})$ \. with \. $a_{n+i}=a_i^{-1}$ \ts
for all \ts $i\in [n]$.
Define the \defn{antidiagonal matrix} \ts ${\rm D}_n=(d_{ij})_{i,j\in[n]}$ \ts is defined by
\ts $d_{ij}=1$ \ts if $i+j=n+1$, and \ts $d_{ij}= 0$ \ts otherwise.
Finally, denote
\begin{equation}\label{eq:J-mat}
{\rm J}\, := \, {\small \begin{pmatrix}
0 & {\rm D}_n\\
-{\rm D}_n & 0
\end{pmatrix}}.
\end{equation}

 Representing isotropic flags $F_{\bullet}$ with matrices $\omega$, the isotropic condition
 translates to
\begin{equation}\label{eq:spForm}
\left\{
\omega\cdot {\rm J}\cdot \omega^T \. = \. {\rm J}
\right\}.
\end{equation}

\smallskip
Consider spaces $E_i:=\langle e_n,\ldots,e_{i}\rangle$ for $i\in[n]$. We extend this chain to a complete flag $E_{\bullet}$ by setting
$$
E_{\ol{i}} \. := \. E_{i+1}^{\perp} \. = \. \langle e_n,\ldots,e_1,e_{\ol{1}},\ldots,e_{\ol{i}}\rangle$$
for each $i\in[n]$.
For all \ts $w\in{\mathcal W}_n$, we can explicitly define Schubert cells in type~$C$ as follows:
\[
\Omega_w^{C}(E_{\bullet}) \, = \, \big\{\ts F_{\bullet} \ : \ \dim(F_i\cap E_j)= k_w(\ol{i},j) \
\text { for } \ i\in[n], \, j\in[\ol{n},n], \. j \ne 0\ts\big\},
\]
where $k_w(\ol{i},j):=\#\{r\leq \ol{i} \. : \. w_r\geq j\}$.
Then the Schubert variety in type~$C$ is an orbit closure which can be defined as the follows:
$$
X_w^{C}(E_{\bullet}) \, = \, \ol{\Omega_w^{C}}(E_{\bullet})\ts.
$$

To construct a generic element $g\in{\sf G}$, we use a generalization of the \emph{Cayley transform},
see \cite[$\S$10]{We66}.\footnote{This construction was suggested to us by David Speyer (personal communication).}
Let $M=(m_{ij})$ be a $2n\times 2n$ matrix such that $M{\rm J}=-{\rm J}M^T$, i.e., $M\in{\fg}$. These are precisely block matrices
  \begin{align}\label{eq:blockskew}
  M=
      \left(\begin{array}{ c  c }
    A & B \\
    C & -A^T
  \end{array}\right), \text{
  where $B=B^T$ and $C=C^T$. }
  \end{align}
  Equivalently, $m_{ij} = m_{(2n+1-j)(2n+1-i)}$ if $i \leq n < j$ or $j \leq n < i$ and $m_{ij} = -m_{(2n+1-j)(2n+1-i)}$ otherwise.
  Thus $M$ can be expressed in terms of those entries weakly above the antidiagonal.

Consider $g=({\rm I}_{2n}+M)^{-1}({\rm I}_{2n}-M)$, where $g=(g_{ij})$. It is straightforward to check $g$ satisfies Equation~\eqref{eq:spForm} since $M{\rm J}=-{\rm J}M^T$ and that such elements $g$ are dense in ${\sf G}$.
 Thus we may represent a generic element of ${\sf G}$ as a $2n\times 2n$ matrix $g=(g_{ij})$, where $g=({\rm I}_{2n}+M)^{-1}({\rm I}_{2n}-M)$ constructed as above. Here we treat the entries $m_{ij}$ for $i+j\leq 2n+1$ as parameters. We treat $g_{ij}$ for $i,j\in[2n]$ as variables constrained by $g\cdot({\rm I}_{2n}+M)=({\rm I}_{2n}-M)$.

 Using this construction, we obtain generic elements $\pi=(\pi_{ij})$ and $\rho=(\rho_{ij})$ of ${\sf G}$, where $\pi\cdot ({\rm I}_{2n}+Y)=({\rm I}_{2n}-Y)$ and $\rho\cdot ({\rm I}_{2n}+Z)=({\rm I}_{2n}-Z)$
for $Y,Z$ matrices of the form of Equation~\eqref{eq:blockskew} with entries in parameters $\by=\{y_{1},\dots,y_{t}\}$ and $\bz=\{z_{1},\dots,z_{t}\}$, respectively. Here $t=2n^2+n$.  Let $\bpi=\{\pi_{ij}\}_{i,j\in[2n]}$ $ \brho=\{\rho_{ij}\}_{i,j\in[2n]}$ be sets of variables.
Take \ts $F_{\bullet}:=\pi E_{\bullet}$ \ts and \ts $G_{\bullet}:=\rho E_{\bullet}$.

Fix three elements \. $u,v,w\in {\mathcal{W}}_n\ts$.
We consider the system for solving
$$X_{{u}}^{C}(F_{\bullet}) \. \cap \. X_{{v}}^{C}(G_{\bullet}) \. \cap \. X_{{w}}^{C}(E_{\bullet}).
$$
The \defn{Stiefel coordinates} \ts $\mathcal{X}_{{w}}^{C}$ \ts for the
Schubert cell \ts $\Omega_{{w}}^{C}$,
are the collection of $2n\times 2n$ matrices $(m_{ij})$ such that
\begin{align}
m_{ij}=
    \begin{cases}
        1  & \text{ if } i={{w}}_j\\
        0  & \text{ if } i<w_j \text{ or } w^{-1}_i<j \\
        x_{ij}  & \text{ otherwise,}\\
    \end{cases}
\end{align}
and such that Equation~\eqref{eq:spForm} is satisfied.

 For convenience,
we relabel matrix entries
$\{x_{ij}\}$ as $\bx=\{x_1,\ldots, x_{s}\}$, reading across rows, top to bottom. To every \ts $\omega\in \mathcal{X}_{{w}}^{C}$ with entries ${\bx}$, we associate the flag given by
the span of the column vectors $e_i(\bx)$ of $\omega$.
%
%
%
Take
\begin{align*}
    \bal \. &:= \. \{\alpha_{ij}: j <i, \ \text{ where } \ {u}_j<{u}_i\},\\
    \bbe \. &:= \. \{\beta_{ij}: j <i,  \text{ where } \ {v}_j<{v}_i\}.
\end{align*}
For $i\in[\ol{n},n]$, \. $i\ne 0$, define the $2n$-vectors \ts $g_i(\bx,\bal)$ \ts and \ts $h_i(\bx,\bbe)$ \ts with entries in  \ts $\mathbb{Z}[\bx,\bal,\bbe]$  \ts such that
\begin{align*}
    g_i(\bx,\bal) \, &  := \, e_i(\bx) \. + \. \sum_{\substack{j <i\\ {u}_j<{u}_i}}\. \alpha_{ij} \. e_j(\bx), \\
    h_i(\bx,\bbe) \, & := \, e_i(\bx) \. + \. \sum_{\substack{j <i \\ {v}_j<{v}_i}} \. \beta_{ij} \. e_j(\bx).
\end{align*}

\smallskip

\begin{prop}\label{prop:liftIntBC}
    Let \ts $\Phi_{\bullet}\in {\sf G}/{\sf B}$ \ts be in general position with \ts $E_{\bullet}$,
    where \ts  $\Phi_i=\langle f_1,\ldots, f_i\rangle$ \ts for $i\in[\ol{n},n]$, \. $i\ne 0$.
    Take \ts ${w}\in {\mathcal{W}}$.
    Then \. $\Phi_{\bullet}\in \Omega_{w}^{C}(E_{\bullet})$ \. \underline{if and only if} \.
    for each \ts for $k\in[\ol{n},n]$, \. $k\ne 0$, there exist unique \ts $\{\beta_{jk}\}$ \ts such that
    \[
    g_k \, = \, f_k \. + \. \sum_{\substack{j <i \\ {w}_j<{w}_i}}\. \beta_{jk} \. f_j \, \in \, E_{{w}_k}\. - \. E_{{w}_k+1}\..
    \]
\end{prop}
\begin{proof}
    This follows from precisely the same argument as in Proposition~\ref{prop:liftInt}.
\end{proof}
\smallskip

\subsection{The construction in type~$C$} \label{ss:type-C-construction}
We consider \ts $\omega\in \mathcal{X}_{{w}}^{C}$ with  column vectors $e_j(\bx)$.  By construction, the corresponding flag $\omega_{\bullet}$ lies in $\Omega_w^{C}(E_{\bullet})$ after imposing Equation~\eqref{eq:spForm}. Then construct $g_i(\bx,\bal)$ and $h_i(\bx,\bbe)$ in terms of $e_j(\bx)$.

By Proposition~\ref{prop:liftIntBC}, $\omega_{\bullet}\in \Omega_u^{C}(F_{\bullet})$ if and only if $i\in[\ol{n},n]$, \. $i \ne 0$, we have \. $g_i(\bx,\bal)\in F_{{u}_i}-F_{{u}_i+1}$ for $\bal$ unique.
To impose this condition,
we require
\begin{align}\label{eq:mainEq1C}
    (\pi_{j1} \hspace*{0.2cm} \pi_{j2}\hspace*{0.2cm}\cdots\hspace*{0.2cm} \pi_{j,2n})\cdot g_i(\bx,\bal) \. = \.0\ts, \text{ for each \ts $j<{u}_i$\ts.}
\end{align}

Similarly, $\omega_{\bullet}\in \Omega_v^{C}(G_{\bullet})$ if and only if for every \. $i\in[\ol{n},n]$, \. $i \ne 0$, we have \.
$h_i(\bx,\bbe)\in G_{{v}_i}-G_{{v}_i+1}$ for $\bbe$ unique.
For this last condition, we need
\begin{align}\label{eq:mainEq2C}
  (\rho_{j1} \hspace*{0.2cm} \rho_{j2}\hspace*{0.2cm}\cdots\hspace*{0.2cm} \rho_{j,2n})\cdot h_i(\bx,\bbe)\. = \.0\ts \text{, for each \ts $j<{v}_i$\ts.}
\end{align}


Let \. $\mathcal{S}^{C}(u,v,w)$ \. denote the system given by Equations~\eqref{eq:spForm},
\eqref{eq:mainEq1C}, and~\eqref{eq:mainEq2C} along with the constraints $\pi=(1+Y)^{-1}(1-Y)$ and $\rho=(1+Z)^{-1}(1-Z)$.
Equation~\eqref{eq:spForm} introduces $12n^2$ quadratic equations in fewer than $4n^2$ variables.
Equations~\eqref{eq:mainEq1C} and~\eqref{eq:mainEq2C} add at most $2\binom{2n}{2}$ trilinear equations and $2\binom{2n}{2}+8 n^2$ variables with $0, 1$ coefficients. The remaining constraints on $\pi$ and $\rho$ each introduce $4n^2$ linear equations with coefficients in $\mathbb{Z}[\by,\bz]$.


\smallskip

Let \ts $c_{u,v}^w(C)$ \ts denote the Schubert coefficients in type~$C$,
defined as structure constants of Schubert polynomials given by the Algebraic
Definition in~$\S$\ref{ss:intro-Schub} or a combinatorial construction in~\cite{ST23}.
Alternatively, Schubert coefficients in type~$C$ can be
defined by Equation~\eqref{eq:SchubStructureCoh}.

\smallskip

\begin{prop}\label{prop:schubHNredC}
  The system \. $\mathcal{S}^{C}(u,v,w_\circ w)$ \. has \. $c_{u,v}^w(C)$ \.
  solutions over \. $\overline{\mathbb{C}(\by,\bz)}$.
\end{prop}

\smallskip

The proof follows verbatim the proof of Proposition~\ref{prop:schubHNredA}.


\medskip

\subsection{The construction in type~$B$} \label{ss:type-B-construction}
In type~$B$, we have \ts ${\sf G}={\sf SO}_{2n+1}(\cc)$ \ts defined analogously.
Let \. $c_{u,v}^w(B)$ \. denote the Schubert coefficients in type~$B$,
defined as above.  For \. $w\in \mathcal{W}_n$ \.
let \.
$$
s(w)\, := \, \big\{i\in[n] \, : \, w_i<0 \big\}.$$
Finally, define
$$
\mathcal{S}^{B}(u,v,w)\. := \. \mathcal{S}^{C}(u,v,w).
$$

\smallskip

\begin{prop}\label{prop:schubHNredB}
  The system \. $\mathcal{S}^{B}(u,v,w_\circ w)$ \. has \.
  $2^{s(u)+s(v)-s(w)}c_{u,v}^w(B)$ \. solutions over \. $\overline{\mathbb{C}(\by,\bz)}$.
\end{prop}\begin{proof}
    By \cite{BH95}, we have
    \begin{equation}\label{eq:Schub-BC}
    c_{u,v}^w(B) \, = \, 2^{s(w)-s(u)-s(v)} \. c_{u,v}^w(C)\ts.
    \end{equation}
    Thus the result follows by Proposition~\ref{prop:schubHNredC}.
\end{proof}


\medskip

\section{Proof of Main Lemma~\ref{l:SV-HNP}}
%
%
Fix type \ts $Y\in\{A,B,C\}$.  Let \ts $u,v,w\in \mathcal W$ \ts be elements in
the corresponding Weyl group.
%
%
%
%
Consider the system \. $\mathcal{S}^{Y}(u,v,w)$ \. defined as in
Sections~\ref{sec:liftA} and~\ref{sec:liftBC}.
By Propositions~\ref{prop:schubHNredA},~\ref{prop:schubHNredC}, and~\ref{prop:schubHNredB},
the number of solutions of the system \. $\mathcal{S}^{Y}(u,v,w)$ \. is equal to the
corresponding Schubert coefficient \. $c_{u,v}^w(Y)$.

Observe that the size of \. $\mathcal{S}^{Y}(u,v,w)$ \.
is \ts $O(n^2)$ \ts in each case.  Thus, these systems are instances of \ts $\HNP$ \ts
by construction, with sets of variables \ts $(\bx,\bal,\bbe,\bpi,\brho)$ \ts and
parameters \ts $(\by,\bz)$.  Therefore, \ts $\neg\SV$ \ts is in \ts $\HNP$ \ts
in each case, as desired.  \qed

\medskip

{\small

\section{Final remarks} \label{s:finrem}

\subsection{} \label{ss:finrem-BW}
Since the Billey--Vakil algebraic system discussed in~$\S$\ref{sss:special-alg}
is the most advanced effort in the direction of this work, it is
worth elaborating as to why it cannot be used to make a reduction to $\HNP$.
There are three major issues, in fact.   First, note that system of constraints
\cite[Eq.~(8)]{BV08} on the variables to be generic has exponentially many equations.
Second, these equations are of the form \ts $\det(X)=0$, thus have exponentially
many terms.  Finally, the additional condition of the set of solutions being
$0$-dimensional is also quite delicate.  It has also been studied by Koiran
\cite{Koiran97} in the context of $\HN$, but the previous two issues again apply in
this case.

\subsection{} \label{ss:finrem-Pur}
Recall Purbhoo's algebraic system discussed in~$\S$\ref{sss:special-alg}, and given
explicitly in Lemma~\ref{lem:Pur}.  This system also has exponentially many constraints
of the form \ts $\det(X) \ne 0$, at least one of which has to be satisfied.
Each of these can be easily converted into a polynomial equation of the form
\ts  $\det(X)\cdot y = 1$ \ts with an auxiliary variable (cf.\ $\S$\ref{sec:appBorel}).
Unfortunately, these equations cannot be used to make a reduction to $\HNP$
for the same reason as above, since they have exponentially many terms.
This issue also appears in our original approach to type~$D$ that is given
in Appendix~\ref{sec:liftD}, see Remark~\ref{rem:type-D}.

\subsection{} \label{ss:finrem-complexity}
Note that Conjecture~\ref{conj:main-Schubert-vanishing-coNP} would follow
if the Knutson and Zinn-Justin sufficient condition for non-vanishing was shown
$\NP$-complete.  Similarly, it would be interesting to see if the Purbhoo
sufficient conditions for nonvanishing
and the
Billey--Vakil sufficient conditions for vanishing given
in~$\S$\ref{sss:special-NP} are $\NP$-complete and $\coNP$-complete, respectively. Note these results
by themselves would not imply Conjecture~\ref{conj:main-Schubert-vanishing-coNP}.

\subsection{} \label{ss:finrem-versions}
Compared to the original draft of this paper \cite{PR24o}, this version includes a new
Appendix~\ref{app:typeD}.\footnote{The extended abstract of \cite{PR24o} is to appear
in \emph{Proc.\ 57-th ACM STOC} (2025), 12~pp. }
The result in type~$D$ and for the BSS model are moved outside
of the main body of the paper to Appendix~\ref{sec:liftD} and~\ref{app:BSS},
respectively, while the example that used to be in the appendix is now
moved to~$\S$\ref{app:Schub}.  In a short companion paper \cite{PR24b},
we have a discussion of combinatorial and philosophical implications of
Main Theorem~\ref{t:main-AM}  in the context of
various number theoretic and complexity theoretic assumptions.

\subsection{} \label{ss:finrem-kInt}
Note that Main Theorem~\ref{t:main-AM} can be generalized in several directions.
Notably, the theorem can be extended to the vanishing problem of generic $k$-fold intersection
\[X_{u_1}\big(E^{(1)}_\bullet\big)\cap X_{u_2}\big(E^{(2)}_\bullet\big)\cap \. \dots \. \cap X_{u_k}\big(E^{(k)}_\bullet\big),\]
for every fixed $k$. In a different direction, Main Theorem~\ref{t:main-AM} can be
extended to enriched cohomology theories.\footnote{This work is in preparation.}

\vskip.7cm

\subsection*{Acknowledgements}
We are grateful to Kevin Purbhoo for sharing his insights which partially
motivated this paper.  We thank Sara Billey, Nickolas Hein, Minyoung Jeon,
Allen Knutson, Leonardo Mihalcea, Greta Panova, Oliver Pechenik,
Maurice Rojas, Mahsa Shirmohammadi, Frank Sottile, Avi Wigderson,
James Worrell, Weihong Xu, Alex Yong and Paul Zinn-Justin for interesting
discussions and helpful comments.  We are especially grateful to
David Speyer for suggesting we use the Cayley transform in types~$C/D$,
and for his insights leading to Appendix~\ref{app:typeD}.  We are thankful
to anonymous STOC reviewers for the careful reading of the paper
and helpful remarks.

This paper was written when the first author was a member at the
Institute of Advanced Study in Princeton,~NJ.  We are grateful for the hospitality.
The first author was partially supported by the NSF grant CCF-2302173.
The second author was partially supported by the NSF MSPRF grant DMS-2302279.
}


\newpage

\appendix



\section{Lifted formulation in type D}\label{sec:liftD}

Recall that our original approach to Main Theorem~\ref{t:main-AM} does
not extend to type~$D$.  An approach that does is given in Appendix~\ref{app:typeD}.
Nevertheless we decided to include a lifted formulation in type~$D$ similar to our
lifted formulations in types $A$, $B$ and~$C$.  There are several reasons for doing that.

First, we wanted to highlight the technical issue that failed the proof in this case.
Second, we use this construction in the BSS model given in Appendix~\ref{app:BSS}.
Even if the new construction Appendix~\ref{app:typeD} would also work,
this lifted formulation is more concise and is of independent interest.
Finally, it is worth expounding on a construction in each type before
presenting a uniform construction in all types.

\subsection{Definitions} \label{ss:type-D-def}
As in type~$C$, let \. $e_{\ol{n}},\ldots,e_{\ol{1}},e_1,\ldots,e_n$ \.
denote a basis in \ts $\mathbb{C}^{2n}$. Define the \defn{symmetric}
bilinear form on \ts $\mathbb{C}^{2n}$ \ts such that \.
$\langle e_i, e_j\rangle=\langle e_{\ol{i}}, e_{\ol{j}}\rangle=0$ \. and \. $\langle e_i, e_{\ol{j}}\rangle=\langle e_{\ol{j}}, e_{i}\rangle=\delta_{ij}$ for $i,j\in[n]$. As before, we use \ts $\ol{k}:=-k$ for $k\in[n]$.

In type~$D$, we have \. ${\sf G}={\sf SO}_{2n}(\cc)$ \. is the special orthogonal
group with respect to the bilinear form \ts $\langle \cdot,\cdot\rangle$.
The corresponding Weyl group \ts $\mathcal{W}_n^{+}$ \ts
is a subgroup of index two in the hyperoctahedral group \ts $\mathcal{W}_n\ts$
It can be viewed the group of signed permutations of $n$ letters with an
even number of sign changes.

An \defn{isotropic flag} \ts $F_\bullet$ is a flag of spaces
\[
\zero \subset F_{n-1}\subset F_{n-1}\subset \dots\subset F_{1}\subset F_{0}\subset \mathbb{C}^{2n},
\]
where \ts $\dim F_i=n-i$ for all $i\in[n]$, and each $F_i$ is \defn{isotropic}, i.e.\ $F_i\subseteq F_i^{\perp}$.

In type~$D$, we identify the flag variety \ts ${\sf G}/{\sf B}$ \ts
with the set of those isotropic flags $F_{\bullet}$ such that $E_0\cap F_0$
is even dimensional.

Recall that \emph{torus subgroup} \ts ${\sf T}$ \ts in this case consists
of diagonal matrices $T$ given by
\ts ${\sf diag}(T)=(a_1,\ldots,a_{2n})$ \. with \. $a_{n+i}=a_i^{-1}$ \ts
for all \ts $i\in [n]$.
Now set ${\rm J}={\rm D}_{2n}$, the antidiagonal $2n\times 2n$ matrix.
%

 Representing flags $F_\bullet$ with matrices $\omega$, the isotropic condition
 translates to\footnote{Equation \. $\omega\cdot {\rm J}\cdot \omega^T \. = \. {\rm J}$ \.
gives \. $\det(\om) \in \{\pm 1\}$.  Ensuring that \. $\det(\om)=1$ \. is
a non-local parity condition that is unavoidable in this setup.}
\begin{equation}\label{eq:soForm}
\left\{\aligned
\omega\cdot {\rm J}\cdot \omega^T \. &= \. {\rm J}\\
\det(\omega) \. &= \. 1\ts.
\endaligned\right.
\end{equation}

Consider spaces $E_i:=\langle e_n,\ldots,e_{i}\rangle$ for $i\in[n]$. We extend this chain to a complete flag $E_{\bullet}$ by setting
\[
E_{\ol{i}}\. := \. E_{i+1}^{\perp} \. = \. \langle e_n,\ldots,e_1,e_{\ol{1}},\ldots,e_{\ol{i}}\rangle
\]
for each $i\in[n]$. We take \ts $E_{n}:=\zero$. Taking $w\in{\mathcal W}_n^+$, we can explicitly define Schubert cells in type~$D:$
\[
\Omega_w^{D}(E_{\bullet})\, = \, \big\{F_{\bullet} \ : \ \dim(F_i\cap E_j)= k_w(\ol{i},j)
 \text { for } \ 0\leq i\leq n-1, \. j\in[\ol{n},n]\cup\{\ol{0},0\}\big\},
\]
where \. $k_w(\ol{i},j):=\#\{r< \ol{i} \. : \. w_r> j\}$.
Then the Schubert variety in type~$D$ is the orbit closure which can
be defined as follows:
$$
X_w^{D}(E_{\bullet})\. = \. \ol{\Omega_w^{D}}(E_{\bullet}).
$$

To construct a generic $g\in{\sf G}$, we again use a generalization of the Cayley transform,
see \cite[$\S$10]{We66}, just as in Section~\ref{sec:liftBC}.
Let $M=(m_{ij})$ be a $2n\times 2n$ matrix such that $M{\rm J}=-{\rm J}M^T$, i.e., $M\in{\fg}$. These are precisely the matrices $ M$ such that
\begin{equation}\label{eq:skew}
    m_{ij} = - m_{(2n+1-j)(2n+1-i)}.
\end{equation}
  Now $M$ can be expressed in terms of those entries strictly above the antidiagonal.

Consider $g=({\rm I}_{2n}+M)^{-1}({\rm I}_{2n}-M)$, where $g=(g_{ij})$. It is straightforward to check $g$ satisfies Equation~\eqref{eq:spForm} since $M{\rm J}=-{\rm J}M^T$ and that such elements $g$ are dense in ${\sf G}$.
 Thus we may represent a generic element of ${\sf G}$ as a $2n\times 2n$ matrix $g=(g_{ij})$, where $({\rm I}_{2n}+M)\cdot g=({\rm I}_{2n}-M)$ constructed as above. Here we treat $m_{ij}$ where $i+j\leq 2n+1$ as parameters. We treat $g_{ij}$ as variables constrained by $({\rm I}_{2n}+M)\cdot g=({\rm I}_{2n}-M)$.

 Using this construction, we obtain generic elements $\pi=(\pi_{ij})$ and $\rho=(\rho_{ij})$ of ${\sf G}$, where $({\rm I}_{2n}+Y)\cdot \pi=({\rm I}_{2n}-Y)$ and $({\rm I}_{2n}+Z)\cdot \rho=({\rm I}_{2n}-Z)$
for $Y,Z$ matrices in parameters $\by=\{y_{1},\dots,y_{t}\}$ and $\bz=\{z_{1},\dots,z_{t}\}$, respectively, satisfying Equation~\eqref{eq:skew} Here $t=2n^2+n$. Let $\bpi=\{\pi_{ij}\}_{i,j\in[2n]}$ $ \brho=\{\rho_{ij}\}_{i,j\in[2n]}$ be sets of variables.
Take \ts $F_{\bullet}:=\pi E_{\bullet}$ \ts and \ts $G_{\bullet}:=\rho E_{\bullet}$.



Pick $u,v,w\in {\mathcal{W}}_n^+$. We consider the system for solving
$$
X_{{u}}^{D}(F_{\bullet}) \. \cap \. X_{{v}}^{D}(G_{\bullet}) \. \cap \. X_{{w}}^{D}(E_{\bullet})\ts.
$$
The \defn{Stiefel coordinates} \ts $\mathcal{X}_{{w}}^{D}$ \ts for the Schubert cell
\ts $\Omega_{{w}}^{D}$, are those $2n\times 2n$ matrices $(m_{ij})$ such that
\begin{align}
m_{ij}=
    \begin{cases}
        1  & \text{ if } i={{w}}_j\\
        0  & \text{ if } i<w_j \text{ or } w^{-1}_i<j \\
        x_{ij}  & \text{ otherwise,}\\
    \end{cases}
\end{align}
and such that Equation~\eqref{eq:soForm} is satisfied.

We relabel matrix entries \ts $\{x_{ij}\}$ \ts as \ts $\bx=\{x_1,\ldots, x_{s}\}$, \. $s\le \binom{2n}{2}$,
reading across rows, top to bottom.  To every \ts $\omega\in \mathcal{X}_{{w}}^{D}$ with entries ${\bx}$, we associate the flag given by
the span of the column vectors $e_i(\bx)$ of $\omega$.
Take
\begin{align*}
    \bal \. &:= \. \{\alpha_{ij} \. : \. j <i,  \ \text{ where } \ {u}_j<{u}_i\}, \\
    \bbe \. &:= \. \{\beta_{ij} \. : \. j <i, \ \text{ where } \ {v}_j<{v}_i\}.
\end{align*}
For $i\in[\ol{n},n]$, \. $i\ne 0$, define the $2n$-vectors \ts $g_i(\bx,\bal)$ \ts and \ts $h_i(\bx,\bbe)$ \ts with entries in  \ts $\mathbb{Z}[\bx,\bal,\bbe]$  \ts such that
\begin{align*}
    g_i(\bx,\bal)\. &:=\. e_i(\bx)\. + \. \sum_{\substack{j <i \\ {u}_j<{u}_i}}\. \alpha_{ij} \. e_j(\bx)\ts, \\
    h_i(\bx,\bbe)\. &:= \. e_i(\bx)\. + \. \sum_{\substack{j <i \\ {v}_j<{v}_i}}\. \beta_{ij}\. e_j(\bx)\ts.
\end{align*}

\smallskip

\subsection{The construction in type~$D$} \label{ss:type-D-construction}

We consider \ts $\omega\in \mathcal{X}_{{w}}^{D}$ with  column vectors $e_i(\bx)$.  By construction, the corresponding flag $\omega_{\bullet}$ lies in $\Omega_w^{D}(E_{\bullet})$ after imposing Equation~\eqref{eq:soForm}. Then construct $g_i(\bx,\bal)$ and $h_i(\bx,\bbe)$.

Extending Proposition~\ref{prop:liftIntBC}, $\omega_{\bullet}\in \Omega_u^{D}(F_{\bullet})$ if and only if $i\in[\ol{n},n]$, \. $i \ne 0$, we have \. $g_i(\bx,\bal)\in F_{{u}_i}-F_{{u}_i+1}$ for $\bal$ unique.
To impose this condition,
we require
\begin{align}\label{eq:mainEq1D}
    (\pi_{j1} \hspace*{0.2cm} \pi_{j2}\hspace*{0.2cm}\cdots\hspace*{0.2cm} \pi_{j,2n})\cdot g_i(\bx,\bal) \. = \.0\ts, \text{ for each \ts $j<{u}_i$\ts.}
\end{align}

Similarly, $\omega_{\bullet}\in \Omega_v^{D}(G_{\bullet})$ if and only if for every \. $i\in[\ol{n},n]$, \. $i \ne 0$, we have \.
$h_i(\bx,\bbe)\in G_{{v}_i}-G_{{v}_i+1}$ for $\bbe$ unique.
For this last condition, we need
\begin{align}\label{eq:mainEq2D}
  (\rho_{j1} \hspace*{0.2cm} \rho_{j2}\hspace*{0.2cm}\cdots\hspace*{0.2cm} \rho_{j,2n})\cdot h_i(\bx,\bbe)\. = \.0\ts \text{, for each \ts $j<{v}_i$\ts.}
\end{align}

%

Let $\mathcal{S}^{D}(u,v,w)$ denote the system given by Equations~\eqref{eq:soForm},~\eqref{eq:mainEq1D}, and~\eqref{eq:mainEq2D} along with the constraints $(1+Y)\cdot\pi=(1-Y)$ and $(1+Z)\cdot\rho=(1-Z)$.
Equation~\eqref{eq:soForm} introduces $12n^2$ quadratic equations and one degree $2n$ equation in fewer than $4n^2$ variables.
Equations~\eqref{eq:mainEq1D} and~\eqref{eq:mainEq2D} add at most $2\binom{2n}{2}$ cubic equations and $2\binom{2n}{2}$ variables with $0, 1$ coefficients.
The remaining constraints on $\pi$ and $\rho$ each introduce $4n^2$ linear equations with coefficients in $\mathbb{Z}[\by,\bz]$.



Let $c_{u,v}^w(D)$ denote the Schubert coefficients in type~$D$, defined as above or as
structure coefficients in Equation~\eqref{eq:SchubStructureCoh} for \ts ${\sf G}={\sf SO}_{2n}$.

\smallskip

\begin{prop}\label{prop:schubHNredD}
  The system \. $\mathcal{S}^{D}(u,v,w_\circ w)$ \. has \. $c_{u,v}^w(D)$ \. solutions over \. $\overline{\mathbb{C}(\by,\bz)}$.
\end{prop}

\smallskip

The proof follows verbatim the proof of Proposition~\ref{prop:schubHNredA}.
\smallskip

\begin{rem}\label{rem:type-D}
    Let us emphasize that the determinantal equation in~\eqref{eq:soForm} gives the system \.
    $\mathcal{S}^{D}(u,v,w)$ \. of exponential size. Thus the proof of Lemma~\ref{l:SV-HNP}
    for types \ts $A$, \ts $B$ and $C$ \ts does not immediately extend to the type~$\ts D$ \ts setting.
    We do this in Appendix~\ref{app:typeD}.
\end{rem}


\medskip

\section{Blum--Shub--Smale model} \label{app:BSS}

In this section we study complexity of the Schubert vanishing problem
in a nonstandard model of computation pioneered by Blum, Shub and Smale.
We removed this section from the main body of the paper to avoid the
destruction from the Main Theorem~\ref{t:main-AM}.  We believe
these results are of independent interest, as they illustrate the
power of computation over fields when it comes to problems in
Algebraic Combinatorics.

\smallskip

\subsection{Computation over fields} \label{ss:BSS}
The \defn{Blum--Shub--Smale} (BSS) \defn{model of computation}
\cite{BSS89, BCSS98} was
introduced to describe computations over general fields~$\kk$.
It can be viewed as analogous to (Boolean) Turing machine,
but in this case the registers that can store arbitrary numbers
from~$\kk$, and rational functions over $\kk$ can be computed
in a single time step.

One can then similarly define complexity classes \ts $\poly_\kk$\ts,
\. $\NP_\kk$\ts, etc.  Although the original paper works largely
over~$\rr$,  many results extend to general fields.
Theorem~\ref{t:intro-BSS} is one of the key results in the area.
%
The counting complexity classes for computations over~$\kk$
were introduced by  B\"urgisser and Cucker \cite{BC06}.
Notably, denote by \ts $\SP_\kk$ \ts the class of counting
functions for the number of accepting paths of problems in \ts $\NP_\kk$.

The relationships between these classes in many ways resemble the classical model.
Notably, B\"urgisser and Cucker prove the following analogue of Toda's theorem (assuming GRH):
$$\SP_\cc \. \subseteq \. \FP_\cc^{\NP_\cc} \ \ \Longrightarrow \ \ \PH \. = \. \Sigmap_4\.,
$$
i.e., the \emph{classical} polynomial hierarchy collapses to the fourth level \cite[Cor.~8.7]{BC06}.

\smallskip

\subsection{Schubert vanishing in BSS model} \label{ss:intro-BSS}
Motivated by the algebro-geometric considerations, we consider the problem
\ts $\HN_\cc$ \ts which asks if the polynomial system \eqref{eq:HN-system}
with coefficients in~$\cc$ has a solution over \ts $\mathbb{C}$.
This is a natural problem in the BSS model
of computation described above.
%
%
One of the key insights of the Blum--Shub--Smale theory is the following
remarkable theorem:

\smallskip

\begin{thm}[{\cite[$\S$5.4]{BCSS98}}{}]\label{t:intro-BSS}
\. $\HN_\cc$ \ts is in \ts $\NP_\cc$.  Moreover, $\HN_\cc$ \ts is \ts $\NP_\cc$-complete.
\end{thm}

\smallskip

This inclusion combined with the system of polynomial equations constructed in the
proof of Lemma~\ref{l:SV-HNP} gives the following result:

\smallskip

\begin{thm} \label{t:main-SV-decision}
\. $\neg\SV$ \. is in \. $\NP_\cc$ \ts in types~$A$, $B$, $C$ and~$D$.
\end{thm}

\smallskip

While the proof in types~$A$, $B$ and~$C$ is straightforward from
the proof of Lemma~\ref{l:SV-HNP}, there is a technical issue in type~$D$.
As we mentioned above, we follow the approach in the lemma,
one of the equations in type~$D$ has exponential size (Remark~\ref{rem:type-D}).
Nevertheless, it can be \emph{evaluated} in poly-time over~$\cc$, and thus can be
used to derive the result.

\smallskip

Our next result is a different kind of inclusion
and also holds in all classical types.

\smallskip

\begin{thm}\label{t:main-SV-BPP}
\. $\neg\SV$ \. is in \. $\P_\rr$ \ts in types~$A$, $B$, $C$ and~$D$.
\end{thm}

\smallskip

The proof of the theorem uses a completely different approach from that in
Lemma~\ref{l:SV-HNP}.  Instead, we use deep results in \cite{Belkale06,Purbhoo06}
to give a linear algebraic formulation of the Schubert vanishing
problem which quickly implies the result.
Curiously, this approach cannot be used to give an
independent proof of Lemma~\ref{l:SV-HNP}.

\begin{rem}\label{r:SV-BPP}
Note that \ts $\NP_\rr$ \ts is likely much more powerful than \ts $\NP_\cc$ \ts
since the former is best known as the \emph{existential theory of the reals}:  \ts
$\exists\rr = \NP_\rr$ \ts and likely not in~$\PH$.  In fact, \ts
$\NP \subseteq\exists\rr \subseteq\PSPACE$ \ts are the best known lower
and upper bounds \cite{Sch10}.  Thus, classes \ts $\NP_\cc$ \ts and
\ts $\P_\rr$ \ts are not easily comparable.
\end{rem}


\medskip

\subsection{Back to Schubert coefficients} \label{ss:intro-BSS-counting}
%
Let \ts $\SP_\cc$ \ts denote the class of counting functions for the number
of accepting paths of problems in \ts $\NP_\cc$.
Denote by \ts $\SHN$ \ts
and \ts  $\SHN_\cc$ \ts the problems of counting solutions of
problems in \ts $\HN$ \ts and \ts $\HN_\cc\ts$, respectively.
By analogy with the Blum--Shub--Smale Theorem~\ref{t:intro-BSS},
B\"urgisser and Cucker prove the following:

\smallskip

\begin{thm}[{\cite[Thm~3.4(i)]{BC06}}{}]\label{t:intro-BC}
\. $\SHN_\cc$ \ts is in \ts $\SP_\cc$.  Moreover, $\SHN_\cc$ \ts is \ts $\SP_\cc$-complete.
\end{thm}

\smallskip

We refer to \cite{BC13} for a thorough treatment of this theory.
Our last result in this model is the following inclusion.

\smallskip

\begin{thm}\label{t:main-Schubert-counting}
\. $\SC$ \. is in \. $\SP_\cc$ \ts in types~$A$, $B$, $C$ and~$D$.
\end{thm}

\smallskip

The proof of this theorem uses the inclusion in Theorem~\ref{t:intro-BC}
and  follows the approach in the proof of Theorem~\ref{t:main-SV-decision}.
Conjecture~\ref{conj:SC-unary} suggests that this result is likely optimal
in the BSS counting hierarchy $\CH_\cc$.

%

\subsection{Proof of  Theorem~\ref{t:main-SV-decision} and Theorem~\ref{t:main-Schubert-counting}}
%
%
 Fix type \. $Y\in\{A,C,D\}$ \. and consider the corresponding system
 \. $\mathcal{S}^{Y}(u,v,w)$.  Take an evaluation of this system at generic values
 \ts $\ov y$ \ts and \ts $\ov z$.  The size of the system is $O(n^2)$.
 Note that in type \ts $D$ \ts we do not expand the determinant and take
 an evaluation directly.

By Propositions~\ref{prop:schubHNredA},~\ref{prop:schubHNredC}, and~\ref{prop:schubHNredD},
deciding if a given flag \ts $\widetilde{F}_\bullet$ \ts is a solution to \ts
$\mathcal{S}^{Y}(u,v,w)$ \ts is in \ts ${\sf NP}_{\mathbb{C}}$ \ts.
Now Theorem~\ref{t:intro-BC} implies the result for all \. $Y\in\{A,C,D\}$.
In type \ts $B$, the result follows from type~$C$ and \eqref{eq:Schub-BC}.
This completes the proof of Theorem~\ref{t:main-SV-decision}.

Finally, Theorem~\ref{t:main-Schubert-counting} follows verbatim the argument
above since the number of solutions of \ts
$\mathcal{S}^{Y}(u,v,w)$ \ts is \emph{equal} to the corresponding
Schubert coefficient \ts $c^w_{u,v}(Y)$ \ts for \. $Y\in\{A,C,D\}$.
In type~$B$, the result follows from type~$C$ and \eqref{eq:Schub-BC}.  \qed

\medskip

\subsection{Proof of Theorem~\ref{t:main-SV-BPP}}\label{sec:LAform}
%
%
Fix type \ts $Y\in\{A,B,C,D\}$ \ts and let \ts ${\sf G}= {\sf G}_Y$ \ts be the
corresponding reductive group.
In each case, ${\sf G}$ is a matrix group lying in an ambient vector space~$V$.
Let ${\sf N}$ \ts denote the subgroup of unipotent matrices, so we have
$$
{\sf N}\subset {\sf B}\subset{\sf G} \subset V.
$$
Let ${\mathfrak n}$ denote the Lie algebra of ${\sf N}$.  We think of
${\mathfrak n}$ as a subspace of~$V$.

For $w\in \mathcal{W}$, define $Z_w:={\mathfrak n}\cap (w{\sf B}_{-}w^{-1})$.
Equivalently, $Z_w$ is the subspace of ${\mathfrak n}$ generated by basis vectors
$e_\al$ for $\al$ a positive root in the root system for ${\sf G}$.

\smallskip

\begin{lemma}\cite[Corollary~2.6]{Purbhoo06}\label{lem:Pur}
Let \. $Y\in\{A,B,C,D\}$.  For generic \ts $\rho,\omega,\tau\in{\sf N}$, we have:
    \[c_{u,v}^w(Y)\. > \. 0 \quad \text{ if and only if } \quad
    \rho Z_u\rho^{-1}+\omega Z_v\omega^{-1}+\tau Z_{w_\circ w}\tau^{-1} \, = \, {\mathfrak n}.
    \]
Here the sum is the usual sum of vector subspaces of~$V$.
\end{lemma}

\smallskip

Following \cite[$\S$17.5]{BCSS98}, define complexity class \ts $\BPP^U_\rr$ \ts to be
the class of problems which can be decided in probabilistic polynomial time by a BSS
machine over~$\rr$.  Here the machine is allowed to pick uniform random numbers in $[0,1]$.
It is known that \ts $\BPP^U_\rr = \P_\rr$\ts, see \cite[$\S$17.6]{BCSS98}.
Thus, it suffices to show that \. $\neg\SV \in \BPP^U_\rr\ts$.


In all types, generate upper-unitriangular matrices \. $\rho,\omega,\tau$ \.
whose entries are uniform random complex numbers \ts $a+i\. b$, where \ts
$a,b\in [-1,1]$ \ts and \ts $i=\sqrt{-1}$.  In type~$A$, this gives a \emph{full measure} on
an $\ep$-ball in ${\sf N}$  around~$\zero_{\sf N}$.  With probability $1$,
all generated numbers will be algebraically independent.  Note that although
these are complex matrices, we will treat them as arrays where entries are two
real numbers, and with multiplication defined as for complex numbers
(in the usual way).

For types $B$, $C$, and~$D$,  we  
then compute the resulting entries constrained by the isotropic condition
in $O(n^3)$ time.  Computing the corresponding inverse matrices using
Cramer's rule, can be done in $O(n^3)$ time in the BSS model.

%

Proceed to pick standard bases in $Z_u,Z_v$, and $Z_{w_\circ w}$,
again minding relevant isotropy conditions.  Compute subspaces as in Lemma~\ref{lem:Pur}.
Check that the sum \.
$\rho Z_u\rho^{-1}+\omega Z_v\omega^{-1}+\tau Z_{w_\circ w}\tau^{-1}$ \.
of these subspaces spans the whole~${\mathfrak n}$.
%
This again can be determined in $O(n^{3})$ time.  By the lemma,
this gives a $\BPP^U_\rr$ algorithm for \ts $\neg\SV$.
%
%
%
\qed


\medskip


\section{A new HN system for all classical types \\ (joint with David E Speyer\protect\footnote{Department of Mathematics,
University of Michigan, Ann Arbor, MI 48109, USA. E-mail: \texttt{speyer@umich.edu}.}\ts\protect\footnote{David Speyer was supported by DMS-2246570 from the Natural Science Foundation.})}
\label{app:typeD}

\smallskip

In this section we prove the following result,
thus completing the proof of Lemma~\ref{l:SV-HNP}:

\smallskip

\begin{lemma}[{\rm cf.\ Lemma~\ref{l:SV-HNP}}{}] \label{l:SV-HNP-app}
\. $\neg\SV$ \. reduces to \. $\HNP$ \. in type~$D$.
\end{lemma}
\smallskip

For the proof, we give an alternative construction of Schubert systems in
all classical types.  This gives a new proof of Lemma~\ref{l:SV-HNP} in types~$A$,
$B$ and~$C$.


\medskip

\subsection{Setup}
Consider ${\sf G}$ a complex Lie group of classical type: \ts \text{\sf GL}$_n(\cc)$, \ts \text{\sf SO}$_{2n+1}(\cc)$, {\sf Sp}$_{2n}(\cc)$, or \ts \text{\sf SO}$_{2n}(\cc)$.
%
%
We can equivalently express Equation~\eqref{eq:SchubStructureInt} as the following:
\begin{align*}
    c_{u,v}^w \,  &= \,\#\big\{X_u(E_{\bullet}) \cap X_v(F_{\bullet}) \cap X_{w_\circ w}(G_{\bullet})\big\}\\
    & = \,\#\big\{ X_u(E_{\bullet})  \cap \pi X_v(E_{\bullet})  \cap \rho X_{w_\circ w}(E_{\bullet})\big\} \\
    & = \,\#\big\{ \Omega_u(E_{\bullet})  \cap \pi \Omega_v(E_{\bullet})  \cap \rho \Omega_{w_\circ w}(E_{\bullet})\big\}
\end{align*}

The first equality follows since can associate generic reference flags $E_{\bullet},F_{\bullet},G_{\bullet}$ in ${\sf G}/{\sf B}$ with generic $\pi$, $\rho \in {\sf G}$.
The second equality follows by Kleiman transversality.
%
%
Using the definition of Schubert cells, we see
\begin{equation}\label{eq:LaxEqEquiv}
    c_{u,v}^w>0 \ \iff \ {{\sf B}_{-}\dot{u}{\sf B}} \ts \cap \ts \pi {{\sf B}_{-}\dot{v}{\sf B}} \ts \cap \ts
    \rho {{\sf B}_{-}\dot{(w_\circ w)}{\sf B}} \. \neq \. \emp
\end{equation}
for generic \ts $\pi, \rho \in {\sf G}$.
Thus we can consider the intersection
\begin{equation}\label{eq:LaxEq}
 \Xi \, := \,  {{\sf B}_{-}\dot{u}{\sf B}} \ts \cap \ts \pi {{\sf B}_{-}\dot{v}{\sf B}} \ts \cap \ts \rho {{\sf B}_{-}\dot{(w_\circ w)}{\sf B}}.
\end{equation}

\smallskip
\begin{rem}
Equation~\eqref{eq:SchubStructureInt} defines Schubert coefficients as the number of points in a $0$-dimensional intersection. The intersection in $\Xi$ is positive dimensional, so the number of points in $\Xi$ is infinite. If desired, one can add additional intersections to this expression in order to cut $\Xi$ down to a finite set of cardinality $c_{uv}^w$. For our purposes, however, $\Xi$ will suffice as is.
\end{rem}

\smallskip

\subsection{Describing generic group elements}\label{sec:appCoeff}
For equations defining the intersection $\Xi$ \ts in Equation~\eqref{eq:LaxEq}, we construct a generic element \ts
$g\in{\sf G}$.

For $\textsf{GL}_n(\cc)$, we generate $g$ as in Section~\ref{sec:liftA}, where $g=(y_{ij})$ a matrix of parameters.
For $\textsf{Sp}_{2n}(\cc)$,  we generate $g$ as in Section~\ref{sec:liftBC}. In particular we create block skew-antisymmetric matrix of parameters $M$, i.e. satisfying Equation~\eqref{eq:blockskew}. Then we create a matrix of variables $g$ and set $g=({\rm I}_{2n}+M)^{-1}({\rm I}_{2n}-M)$.

Similarly, for $\textsf{SO}_{2n}(\cc)$,  we generate $g$ as in Section~\ref{sec:liftD}. In particular we create skew-antisymmetric matrix of parameters $M$, i.e.\  satisfying Equation~\eqref{eq:skew}. Then we create a matrix of variables $g$ and set $g=({\rm I}_{2n+1}+M)^{-1}({\rm I}_{2n+1}-M)$.
For
$\textsf{SO}_{2n+1}(\cc)$, we may generate $g$ similarly to $\textsf{SO}_{2n}(\cc)$, using the same argument as in Section~\ref{sec:liftD}. In particular we create skew-antisymmetric matrix of parameters $M$. Then we create a matrix of variables $g$ and set $g=({\rm I}_{2n+1}+M)^{-1}({\rm I}_{2n+1}-M)$.

We summarize these descriptions in Table~\ref{Tab:Gmain} below.  Here for each ${\sf G}$, we describe the set of parameters $\Par(g)$, the set of variables $\VarB(g)$, and the set of equations $\EG(g)$ needed to ensure $g\in {\sf G}$.  Note in each case that the sizes of $\Par(g)$, $\VarG(g)$, and $\EG(g)$ are in $O(n^2)$.

\begin{table}[ht]
\renewcommand{\arraystretch}{1.5}
\begin{tabular}{ |p{1.4cm}||p{4.5cm}|p{2.5cm}|p{5.2cm}|  }
 \hline
 ${\sf G}$& $\Par(g)$ & $\VarB(g)$ & $\EG(g)$\\
 \hline
 $\textsf{GL}_n$  & $\{m_{ij}\}_{i,j\ts \in \ts [n]}$  & $\emp$   &$\emp$ \\
 $\textsf{SO}_{2n+1}$ & $\{m_{ij}\}_{i+j \ts \leq \ts 2n+2, \ i,j\ts \in \ts [2n+1]}$  &   $\{g_{ij}\}_{i,j\ts \in \ts[2n+1]}$  & $({\rm I}_{2n+1}+M)\cdot g=({\rm I}_{2n+1}-M)$   \\
 $\textsf{Sp}_{2n}$& $\{m_{ij}\}_{i+j\ts \leq \ts 2n+1, \ i,j \ts \in \ts [2n]}$   &   $\{g_{ij}\}_{i,j\ts \in \ts[2n]}$  & $ ({\rm I}_{2n}+M)\cdot g=({\rm I}_{2n}-M)$\\
 $\textsf{SO}_{2n}$ & $\{m_{ij}\}_{i+j\ts \leq \ts 2n+1, \ i,j\ts \in \ts  [2n]}$   &   $\{g_{ij}\}_{i,j\ts \in \ts[2n]}$  & $({\rm I}_{2n}+M)\cdot g=({\rm I}_{2n}-M)$\\
 \hline
\end{tabular}
\bigskip
\caption{Generating a matrix \ts $g\in {\sf G}$.}
  \label{Tab:Gmain}
\end{table}
%

\subsection{Describing generic Borel subgroup elements}\label{sec:appBorel}
To construct equations defining the intersection in Equation~\eqref{eq:LaxEq}, we must describe the equations ensuring that an upper triangular matrix \ts $B=(b_{ij})$ \ts lies \ts ${\sf B}\subset{\sf G}$.

For $\textsf{GL}_n(\cc)$, we have \ts $B=(b_{ij})_{i,j\in[n]}$ \ts lies in ${\sf B}$ if $B$ is invertible. Thus, for $b$ a variable, we have \ts $B\in {\sf B}$ \ts precisely when
\begin{equation}\label{eq:appTypeAborel}
 b\cdot \prod_{i=1}^n b_{ii} \. = \. 1
\end{equation}
is satisfiable.

For $\textsf{SO}_{2n+1}(\cc)$, let \ts ${\rm J}:={\rm D}_{2n+1}$ \ts be the antidiagonal matrix.
Since we require $B\in {\sf G}$, matrix $B$ must satisfy \ts
$B^T \cdot {\rm J}\cdot B= {\rm J}$ and $\det(B) = 1$, cf.\ Equation~\eqref{eq:spForm}.
Imposing $B^T \cdot {\rm J}\cdot B= {\rm J}$ implies that \ts $b_{ii} \ts b_{(2n+2-i)(2n+2-i)}=1$ \ts
for each $i\in[2n+1]$.
Since $B$ is upper triangular, this implies $\det(B) = b_{nn} = \pm 1$. So, in order to have $B$ in $\SO_{2n+1}$, we simply impose that $B^T J B = J$ and $b_{nn} = 1$.

For $\textsf{Sp}_{2n}(\cc)$  set
$${\rm J}\, := \, {\small \begin{pmatrix}
0 & {\rm D}_n\\
-{\rm D}_n & 0
\end{pmatrix}},
$$
cf.\ Equation~\eqref{eq:J-mat}.
To have $B\in {\sf G}$, then $B$ must satisfy Equation~\eqref{eq:spForm}. That is, we impose
$B^T \cdot {\rm J}\cdot B= {\rm J}$.

For \ts $\textsf{SO}_{2n}(\cc)$,  set \ts ${\rm J}:={\rm D}_{2n}$.
To have $B\in {\sf G}$, we have matrix $B$ must satisfy Equation~\eqref{eq:soForm}. That is,
we have \ts $B\cdot {\rm J}\cdot B^T = {\rm J}$ \ts and \ts $\det(B) = 1$.
Imposing \ts $B^T \cdot {\rm J}\cdot B= {\rm J}$ \ts
implies that we have \ts $b_{ii} \ts b_{(2n+1-i)(2n+1-i)}=1$ \ts for each $i\in[2n]$.
Thus, since $B$ is upper triangular, we have \ts $B\cdot {\rm J}\cdot B^T = {\rm J}$ \ts implies \ts $\det(B) = 1$.
Therefore, it suffices to impose $B^T \cdot {\rm J}\cdot B= {\rm J}$.

We summarize in Table~\ref{Tab:Borel}
the description of the set of variables $\VarB(B)$ and set of equations $\EB(B)$ needed to ensure that \ts $B\in {\sf B}\subset {\sf G}$.
Note that in each case that the sizes of $\VarB(B)$ and $\EB(B)$ are $O(n^2)$. Of course, this construction similarly generates $B\in {\sf B}_{-}\subset {\sf G}$ after transposing.

\begin{table}[ht]
\renewcommand{\arraystretch}{1.5}
\begin{tabular}{ |p{1.5cm}||p{3cm}|p{8cm}|  }
 \hline
 ${\sf G}$& $\VarB(B)$ & $\EB(B)$\\
 \hline
 $\textsf{GL}_n$   & $\{b_{ij}\}_{i\ts \leq \ts  j \ts \in \ts [n]}\cup\{b\}$   &$\{ b\cdot \prod_{i\ts \in \ts[n]}b_{ii}=1  \}$ \\
 $\textsf{SO}_{2n+1}$ &   $\{b_{ij}\}_{i\ts \leq \ts  j\ts \in \ts [2n+1]}$  & $\{ b_{ii}\ts b_{(2n+2-i)(2n+2-i)}=1  \}_{i\ts \in \ts [2n+1]}\cup\{b_{nn}=1\}$   \\
 $\textsf{Sp}_{2n}$ &   $\{b_{ij}\}_{i\ts \leq \ts  j\ts \in \ts [2n]}$  & $\{ b_{ii}\ts b_{(2n+2-i)(2n+2-i)}=1  \}_{i\ts \in \ts [2n]}$\\
 $\textsf{SO}_{2n}$    &   $\{b_{ij}\}_{i\ts \leq \ts j\ts \in \ts [2n]}$  & $\{ b_{ii} \ts b_{(2n+2-i)(2n+2-i)}=1  \}_{i\ts \in \ts [2n]}$\\
 \hline
\end{tabular}

\bigskip

\caption{Generating a matrix \ts $B\in{\sf B}$.}
  \label{Tab:Borel}
\end{table}


\smallskip

\subsection{Describing the Equations} \label{Eqs}
Now we describe the equations defining $\Xi$ in~\eqref{eq:LaxEq}.
Observe that \ts $\Xi \neq \emp$ \ts if and only if there exist \ts
$P_1,P_2,P_3\in {\sf B}_{-}$ and $Q_1,Q_2,Q_3\in {\sf B}$ \ts such that
 \[
 P_1\dot{u}Q_1 \. = \.  \pi P_2\dot{v}Q_2 \quad \text{ and } \quad P_1\dot{u}Q_1 \. = \. \rho P_3\dot{(w_\circ w)}Q_3\ts.
 \]
Let \. $\mathcal{E}^{Y}(u,v,w)$, \. $Y \in \{A,B,C,D\}$, denote the system given by the following equations:
\begin{enumerate}
    \item \ $\EG(\pi)\cup \EG(\rho)$ in Table~\ref{Tab:Gmain},
    \item \ $\EB(P_1)\cup\EB(P_2)\cup\EB(P_3)\cup \EB(Q_1)\cup\EB(Q_2)\cup\EB(Q_3)$ \. in Table~\ref{Tab:Borel},
    \item \ $P_1\dot{u}Q_1 = \pi P_2\dot{v}Q_2$ \. and
    \item \ $P_1\dot{u}Q_1 = \rho P_3\dot{(w_\circ w)}Q_3$\ts.
\end{enumerate}
Here $Y$ is the type corresponding to ${\sf G}$.
The system \ts $\mathcal{E}^{Y}(u,v,w)$ \ts has the following variables:
\begin{enumerate}
    \item $\bal:=\VarG(\pi)\cup \VarG(\rho)$ in Table~\ref{Tab:Gmain}, and
    \item $\bbe:=\VarB(P_1)\cup \VarB(P_2)\cup \VarB(P_3)\cup \VarB(Q_1)\cup \VarB(Q_2)\cup \VarB(Q_3)$ \. in Table~\ref{Tab:Borel}.
\end{enumerate}
We solve $\mathcal{E}^{Y}(u,v,w)$ over $\mathbb{C}(\by,\bz)$, where $\by=\Par(\pi)$ and $\bz=\Par(\rho)$ from Table~\ref{Tab:Gmain}.

\begin{rem}
In types $B$ and $C$, one should be careful to choose a representative $\dot{\omega}$ for $\omega\in\mathcal{W}$ with determinant $1$.
In type $C$, we may use a signed permutation matrix where $\omega_{ij}\in\{0,-1 \}$ if $1 \leq i \leq n$ and $n+1 \leq j \leq 2n$, and $\omega_{ij}\in\{ 0, 1 \}$ otherwise. In type $B$, one may do this, or instead set the central value of $\omega_{(n+1)(n+1)}$ to be $\pm 1$ as needed. Note that in type $D$ this property holds for any representative.
\end{rem}





\begin{proof}[Proof of Lemma~\ref{l:SV-HNP-app} and Lemma~\ref{l:SV-HNP}]
Fix type \ts $Y\in\{A,B,C,D\}$.  Let \ts $u,v,w\in \mathcal W$ \ts be elements in
the corresponding Weyl group.
    By Equation~\eqref{eq:LaxEqEquiv}, it follows that $\mathcal{E}^{Y}(u,v,w)$ is satisfiable over $\mathbb{C}(\by,\bz)$ if and only if $c_{u,v}^w>0$. It is straightforward to check $\mathcal{E}^{Y}(u,v,w)$ has size $O(n^2)$.
 Thus, these systems are instances of \ts $\HNP$ \ts
by construction, with sets of variables \ts $(\bal,\bbe)$ \ts and
parameters \ts $(\by,\bz)$.  Therefore, \ts $\neg\SV$ \ts is in \ts $\HNP$ \ts
in each case, as desired.
\end{proof}

\end{document}